\documentclass[11pt,reqno]{amsart}
\usepackage{graphicx,enumerate, amssymb, url} 
\usepackage{dsfont}
\usepackage[left=2.5cm,right=2.5cm,top=2.5cm,bottom=2.5cm,includeheadfoot]{geometry}

\usepackage[colorlinks=true,pdfstartview=FitV,linkcolor=blue,citecolor=blue, urlcolor=blue]{hyperref}

\numberwithin{equation}{section}

\theoremstyle{plain}
\newtheorem{theorem}{Theorem}[section]
\newtheorem{lemma}[theorem]{Lemma}
\newtheorem{proposition}[theorem]{Proposition}
\newtheorem{conjecture}[theorem]{Conjecture}
\newtheorem{corollary}[theorem]{Corollary}

\theoremstyle{definition}
\newtheorem{definition}{Definition}[section]

\theoremstyle{remark}
\newtheorem{remark}{Remark}[section]

\newcommand {\calN}        {{\mathcal N}}
\newcommand {\calF}        {{\mathcal F}}
\newcommand {\calS}        {{\mathcal S}}
\newcommand{\bitem}{\begin{itemize}}
\newcommand{\eitem}{\end{itemize}}
\newcommand{\mc}[1]{\mathcal{#1}}

\newcommand {\Rn}        {{\mathbb R^n}}
\newcommand {\Rm}        {{\mathbb R^m}}
\newcommand{\N}{\mathbb{N}}
\newcommand{\R}{\mathbb{R}}

\newcommand{\EE}{\mathbb{E}}

\newcommand{\bpm}{\begin{pmatrix}}
\newcommand{\epm}{\end{pmatrix}}
\newcommand{\bsm}{\left(\begin{smallmatrix}}
\newcommand{\esm}{\end{smallmatrix}\right)}

\newcommand{\ul}[1]{\underline{#1}}
\newcommand{\ol}[1]{\overline{#1}}

\newcommand{\mrm}[1]{\mathrm{#1}}

\newcommand{\col}[2]{{#1}_{\bullet,#2}}
\newcommand{\veps}{\varepsilon}

\DeclareMathOperator{\supp}{supp}

\title[Weak Recovery Thresholds for TomoPIV]{{C}ritical {P}arameter {V}alues and {R}econstruction
  {P}roperties of {D}iscrete {T}omography: {A}pplication to {E}xperimental
  {F}luid {D}ynamics}
\author[S.~Petra, C.~Schn\"{o}rr, A. Schr\"{o}der]{Stefania Petra, 
Christoph Schn\"{o}rr, Andreas Schr\"{o}der}
\address[S.~Petra, C.~Schn\"{o}rr]{Image and Pattern Analysis Group, University of Heidelberg, Speyerer Str.~6, 69115 Heidelberg, Germany}
\email{\{petra,schnoerr\}@math.uni-heidelberg.de}
\urladdr{iwr.ipa.uni-heidelberg.de}
\address[A.~Schr\"{o}der]{Institute of Aerodynamics and Flow Technology, German Aerospace Center, Bunsenstr.~10, 37073 G\"{o}ttingen, Germany}
\email{andreas.schroeder@dlr.de}
\urladdr{dlr.de/as}
\date{} 
\thanks{Support by the German Research Foundation (DFG) is gratefully acknowledged, grant SCHN457/11.}

\begin{document}

\sloppy

\begin{abstract}

We analyze representative ill-posed scenarios of tomographic PIV with a focus on conditions for unique volume reconstruction. Based on sparse random seedings of a region of interest with small particles, the corresponding systems of linear projection equations are probabilistically analyzed in order to determine
(i) the ability of unique reconstruction in terms of the imaging geometry and the critical sparsity parameter, and 
(ii) sharpness of the transition to non-unique reconstruction with ghost particles when choosing the sparsity parameter improperly. The sparsity parameter directly relates to the seeding density used for PIV in experimental fluids dynamics that is chosen empirically to date. Our results provide a basic mathematical characterization of the PIV volume reconstruction problem that is an essential prerequisite for any algorithm used to actually compute the reconstruction. Moreover, we connect the sparse volume function reconstruction problem from few tomographic projections to 
major developments in compressed sensing.
\end{abstract}

\maketitle

\section{Introduction}

Motivated by an application from fluid dynamics \cite{Els-06}, we investigate
conditions for an highly underdetermined nonnegative system of linear equations
to have a unique nonnegative solution, provided it is sparse.
The sought solution is a sparse 3D image of particles immersed in a fluid
known only from its projection. This projection represents the simultaneous 2D images
captured by few camera sensors from different viewing directions, see Fig. \ref{fig:TomoPIV}.
The reconstruction of the 3D image from the 2D images employs a standard algebraic image reconstruction 
model, which assumes that the image consists of an array of unknowns (cells, voxels), and sets up 
algebraic equations for the unknowns in terms of measured projection data. The latter are the pixel entries
in the recorded 2D images that represent the integration of the original 3D light intensity
distribution along the pixels line-of-sight. The number of cameras is limited to 3 to 6 cameras, typically 4. As a consequence, the reconstruction problem becomes severely ill-posed.

Thus, we consider a huge and severely underdetermined linear system 
\begin{equation} \label{eq:Ax=b}
A x = b,\qquad A \in \R^{m \times n},\qquad
m \ll n,
\end{equation}
with the following properties: a \emph{very sparse} nonnegative measurement matrix $A$ with \emph{constant small support} of length $\ell$ of all column vectors, 
\begin{equation} \label{eq:Axb-properties}
A \geq 0,\; x \geq 0,\qquad \supp(\col{A}{j}) = \ell \ll m,\qquad \forall j = 1,\dotsc,n.
\end{equation}
and a nonnegative $k$-sparse solution vector $x$.
While $\ell$ equals the number of cameras, $k$ is equal related to the particle density 
(equal, in the present work, proportional, in practice). We also consider the discretization 
(or resolution) parameter $d$, and relate it to the number of discretization cells and number of measurements:
\begin{align}
\label{eq:m_n_2D}
m&=\ell\cdot O(d), \quad n=O(d^2), \quad {\rm \ in \ 2D},\\
\label{eq:m_n_3D}
m&=\ell\cdot O(d^2), \quad n=O(d^3), \quad {\rm in \ 3D}\ .
\end{align}

We will answer the following question: what is the maximal number of particles, depending on the image resolution 
parameter $d$, that can be localized uniquely? Formally, we want to relate the \emph{exact} recovery of $x$ from 
it's noiseless measurements $b$ to the sparsity $k$ and to the dimensions of $m,n$ of the projection matrix $A$.
Moreover, we investigate critical values of the sparsity parameter $k$ such that most $k$-sparse nonnegative solutions
are the unique nonnegative solutions of \eqref{eq:Ax=b} with high probability.

\begin{figure}
\centerline{
\includegraphics[height=0.2\textheight]{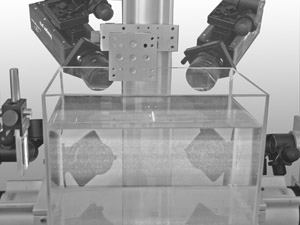} 
\hspace{0.025\textwidth}
\includegraphics[height=0.2\textheight]{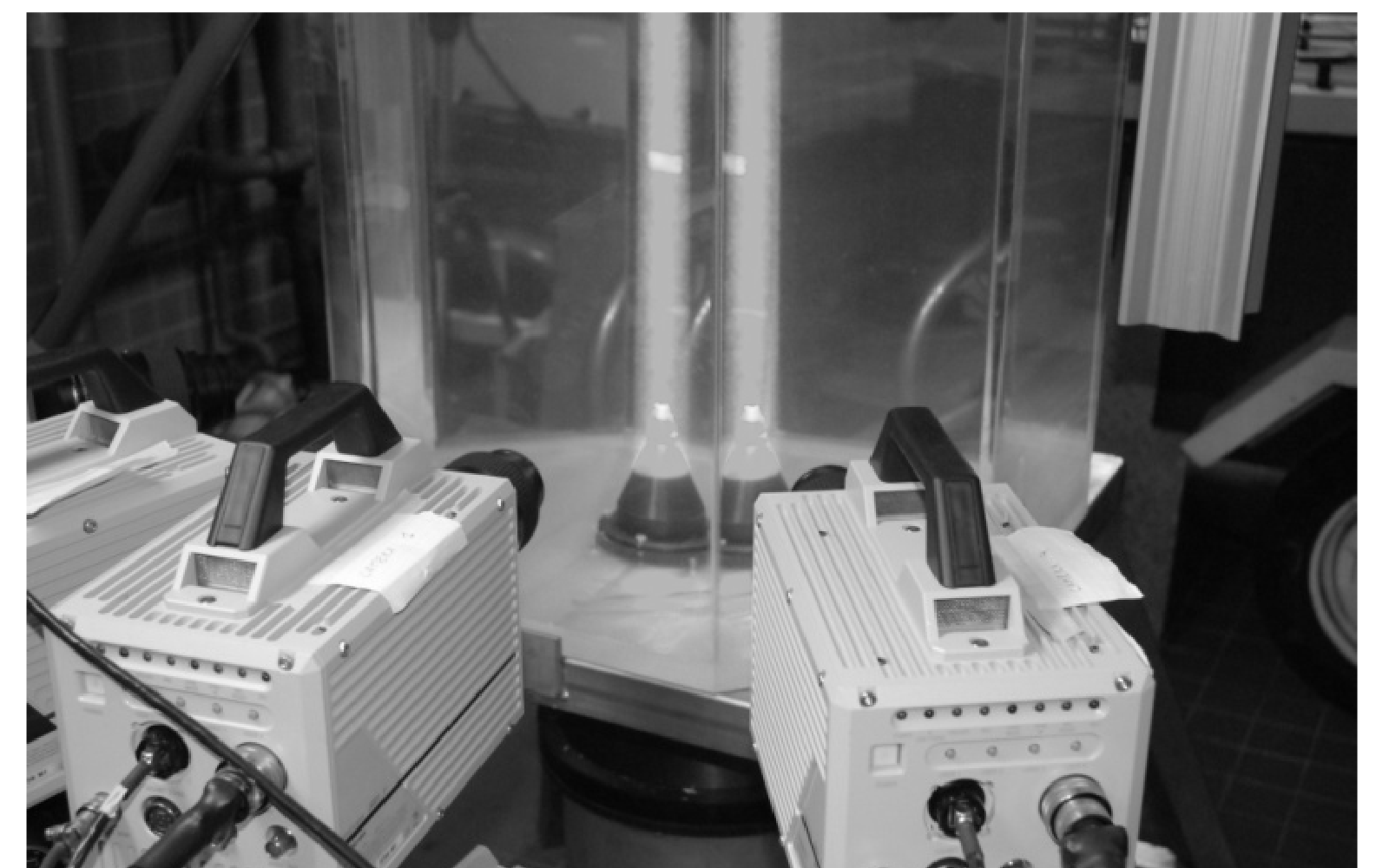}
}
\caption{
Typical camera arrangements: in circular configuration (right) or all in line (left).
}
\label{fig:TomoPIV}
\end{figure}

\subsection{Related Work}
Research on compressed sensing \cite{CompressedSensing-06,Candes-CompressiveSampling-06} focuses on properties of 
underdetermined linear systems that guarantee exact recovery of sparse or compressible signals $x$ from 
measurements $b$. Donoho and Tanner \cite{DonohoT10, DonTan10_Precise} have computed sharp reconstruction thresholds for 
random measurement matrices, such that for given a signal length $n$ and numbers of measurements $m$, the maximal 
sparsity value $k$ which guarantees perfect reconstruction can be determined explicitly.
The authors derived their results by connecting it to problem from geometric probability
that $n$ points in general position in $\Rm$ can be linearly separated \cite{Wendel-62}.
This holds with probability $\Pr(n,m)=1$ for $n/m \leq 1$, and with $\Pr(n,m) \to 1$ if $m \to \infty$ and $1 \leq n/m < 2$,
where
\begin{equation} \label{eq:WendelPR}
\Pr(n,m) = \frac{1}{2^{n-1}} \sum_{i=0}^{m-1} 
\binom{n-1}{i}\ .
\end{equation}
The authors show in \cite[Thm. 1.10]{DonohoT10} that the probability of uniqueness of a $k$-sparse nonnegative vector equals $\Pr(n-m,n-k)$, provided $A$ satisfies certain conditions which do not hold in our considered application.
Likewise, by exploiting again Wendel's theorem, Mangasarian and Recht showed \cite{Man09ProbInteger}
that a binary solution is most likely unique if $m/n>1/2$, provided that $A$ comes from a centrosymmetric distribution.
Unfortunately, the underlying distribution $A$ lacks symmetry with respect to the origin.
However, we recently showed \cite{Petra-Schnoerr-12a} for a three camera scenario that there are thresholds on sparsity 
(i.e. density of the particles), below which exact recovery will succeed and above which it fails with high probability. 
This explicit thresholds depend on the number of measurements (recording pixel in the camera arrays).
The current work investigates further geometries  and focus on an \emph{average case analysis} of conditions under which 
\emph{uniqueness} of $x$ can be expected with \emph{high probability}. A corresponding tail bound implies a weak threshold effect and criterion for adequately choosing the value of the sparsity parameter $k$. 

\subsection{Notation}
$|X|$ denotes the cardinality of a finite set $X$ and $[n] = \{1,2,\dotsc,n\}$ for $n \in \N$. 
We will denote by
$\|x\|_{0} = |\{i \colon x_{i} \neq 0 \}|$ 
and $\R_{k}^{n} = \{ x \in \R^{n} \colon \|x\|_{0} \leq k \}$ the set of $k$-sparse vectors. The corresponding sets of non-negative vectors are denoted by $\R_{+}^{n}$ and $\R_{k,+}^{n}$, respectively. The support
of a vector $x\in \R^{n}$, $\mrm{supp}(x) \subseteq [n]$,
is the set of indices of non-vanishing components of $x$. 

For a finite set $S$, the set $\mc{N}(S)$ denotes the union of all
neighbors of elements of $S$ where the corresponding relation (graph)
will be clear from the context.


$\col{A}{i}$ denotes the $i$-th column vector of a matrix $A$. For
given index sets $I, J$, matrix $A_{I J}$ denotes the submatrix of $A$ 
with rows and columns indexed by $I$ and $J$, respectively. $I^{c},
J^{c}$ denote the respective complement sets. Similarly, $b_{I}$
denotes a subvector of $b$.

$\EE[\cdot]$ denotes the expectation operation applied to a random variable and $\Pr(A)$ the 
probability to observe an event $A$.

\section{Graph Related Properties of Tomographic Projection Matrices}
\label{sec:expander}

Recent trends in compressed sensing \cite{RIP-P-SMM-08, XuHassibi_Expander} tend to replace
random dense matrices by adjacency matrices of ''high quality'' expander graphs.
Explicit constructions of such expanders exist, but are quite involved.
However, random $m\times n$ binary matrices with nonreplicative columns that have
$\lfloor \ell n\rfloor$ entries equal to $1$, perform numerically extremely well, even if $\ell$ is small, as shown in \cite{RIP-P-SMM-08}.
In \cite{HassibiIEEE, Petra2009} it is shown that perturbing the elements of
adjacency matrices of expander graphs with low expansion, can also improve performance.

\subsection{Preliminaries}

For simplicity, we will restrict on situations were the intersection lengths of projection rays corresponding
to each camera with each discretization cell are all equal. Thus, we can make the assumption that the
entries of $A$ are binary. It will be useful to denote the set of cells by $C = [n]$ and the set of rays by $R = [m]$.
The incidence relation between cells and rays is then given by  
\begin{equation} \label{eq:def-A}
(A)_{ij}=
\begin{cases} 1,& 
\quad \text{if $j$-th ray intersects $i$-th cell},\\
 0, & \quad \text{otherwise}, \end{cases}
\end{equation}
for all $i\in[m]$, $j\in[n]$. Thus, cells and rays correspond to columns and rows of $A$.

This gives the equivalent representation in terms of a \emph{bipartite graph} $G = (C,R;E)$ with left and right vertices $C$ and $R$, and edges $(c,r) \in E$ iff $(A)_{rc} = 1$. $G$ has \emph{constant left-degree} $\ell$ equal to the number of projecting directions. 

For any non-negative measurement matrix $A$ and the corresponding graph, the set
\[
\calN(S) = \{i\in[m] \colon A_{ij}>0,\, j\in S\}
\]
contains all neighbors of $S$. The same notation applies to neighbors of subsets $S \subset [m]$ of right nodes.
Further, we will call any non-negative matrix \emph{adjacency matrix}, based on the incidence relation of its non-zero entries.

If $A$ is the non-negative adjacency matrix of a bipartite graph with constant left degree $\ell$, the \textbf{perturbed matrix} $\tilde A$ is computed by uniformly perturbing the non-zero entries $A_{ij} > 0$ to obtain $\tilde A_{ij} \in [A_{ij}-\veps,A_{ij}+\veps]$, and by normalizing subsequently all column vectors of $\tilde A$. In practice, such perturbation can be implemented by discretizing the image by radial basis functions of unequal size or and choose their locations on an irregular grid.

\vspace{0.25cm}
The following class of graphs plays a key role in the present context and in the field of compressed sensing in general.
\begin{definition}\label{def:Expander}
  A $(\nu,\delta)$-unbalanced expander is a bipartite simple graph $G
  = (L,R;E)$ with constant left-degree $\ell$ such that for any $X \subset L$ with
  $|X| \leq \nu$, the set of neighbors $\calN(X) \subset R$ of $X$ has at least size
  $|\calN(X)| \geq \delta \ell |X|$.
\end{definition}

Recovery of a $k$-sparse nonnegative solution via an
$(\nu,\delta)$-unbalanced expander was derived in \cite{WangIEEE}.
It employs the smallest expansion constant $\delta$ with respect to other similar results 
in the literature.

\begin{theorem}\label{thm:wang} Let $A$ be the adjacency
matrix of a $(\nu,\delta)$-unbalanced expander 
and $1 \geq \delta>\frac{\sqrt{5}-1}{2}$.
Then for any $k$-sparse vector $x^*$ with $k\le \frac{\nu}{(1+\delta)}$, the
solution set $\{x \colon Ax=Ax^*,x \ge 0\}$ is a singleton.
\end{theorem}

Now let $A$ denote the tomographic projection matrix, and consider a subset $X \subset C$ of $|X|=k$ columns and a corresponding $k$-sparse vector $x$. Then $b = A x$ has support $\mc{N}(x)$, and we may remove the subset of $\mc{N}(X)^{c} = (\mc{N}(X))^{c}$ rows from the linear system $A x = b$ corresponding to $b_{r}=0,\, \forall r \in R$. Moreover, based on the observation $\mc{N}(X)$, we know that
\begin{equation} \label{eq:reduced-dimensions}
 X \subseteq \mc{N}(\mc{N}(X))
 \qquad\text{and}\qquad
 \mc{N}(\mc{N}(X)^{c}) \cap X = \emptyset.
\end{equation}

We continue by formalizing the system reduction just described.

\begin{definition}
The \emph{reduced system} corresponding to a given non-negative vector $b$,
\begin{equation} \label{eq:red-system}
 A_{red} x = b_{red},\qquad A_{red} \in 
 \R_{+}^{m_{red} \times n_{red}},
\end{equation}
results from $A, b$ by choosing the subsets of rows and columns
\begin{equation} \label{eq:def-RbCb}
 R_{b} := \supp(b),\qquad
 C_{b} := \mc{N}(R_{b}) \setminus \mc{N}(R_{b}^{c})
\end{equation}
with 
\begin{equation} \label{def:mn-red}
 m_{red} := |R_{b}|,\qquad
 n_{red} := |C_{b}|.
\end{equation}
\end{definition}
Note that for a vector $x$ and the bipartite graph induced by the measurement matrix $A$, we have the correspondence (cf.~\eqref{eq:reduced-dimensions})
\[
X = \supp(x),\qquad
R_{b} = \mc{N}(X),\qquad
C_{b} = \mc{N}(\mc{N}(X)) \setminus \mc{N}(\mc{N}(X)^{c}).
\]
We further define
\begin{equation}\label{def:feasSet}
\calS^+:=\{x \colon Ax=b, x\ge 0\}
\end{equation}
and
\begin{equation}\label{def:redfeasSet}
\calS_{red}^+:=\{x \colon A_{R_b C_b}x=b_{R_b}, x\ge 0\}\ .
\end{equation}
The following proposition asserts that solving the reduced system 
\eqref{eq:red-system} will always recover the support of the solution to the original system $A x  = b$
\begin{proposition}\cite[Prop. 5.1]{Petra-Schnoerr-12a}\label{prop:redfeasSet}
Let $A\in \R^{m\times n}$ and $b\in\R^m$ have nonnegative entries only, and let
$\calS^+$ and $\calS_{red}^+$ be defined by \eqref{def:feasSet} and
\eqref{def:redfeasSet}, respectively. Then
\begin{equation}\label{eq:feasSet}
  \calS^+=\{x\in \R^n \colon x_{(C_b)^c}=0\;\text{ and }\; x_{C_b}\in
  \calS_{red}^+\}.
\end{equation}
\end{proposition}

Consequently, we can restrict the linear system $A x = b$ to the subset of columns
$\mc{N}(\mc{N}(X)) \setminus \mc{N}(\mc{N}(X)^{c}) \subset C$, and only consider properties of 
this reduced systems. 

\subsection{Guaranteed Uniqueness}
\label{sec:uniqueness}

Uniqueness of $x\in\Rn_{\delta k,+}$ is guaranteed if all $k$ or less-sparse
supported on ${\rm supp}(x)$ induce overdetermined reduced systems
with $m_{red}/n_{red}> \delta \ell \ge  \frac{\sqrt{5}-1}{2}\ell$.

\begin{proposition}\cite[Th. 3.4]{Petra-Schnoerr-12a}\label{thm:Wang}  Let $A$ be the adjacency
matrix of a bipartite graph such that for all
random subsets $X \subset C$ of $|X| \leq k$ left nodes, the set of neighbors $\calN(X)$ of $X$ satisfies
\begin{equation} \label{eq:condition-Wang}
|\calN(X)| \geq \delta 
\ell |\calN(\calN(X))\setminus \calN(\calN(X)^c)| 
\qquad\text{with}\qquad 
\delta>\frac{\sqrt{5}-1}{2}.
\end{equation}
Then, for any $\delta k$-sparse nonnegative vector $x^*$, the
solution set $\{x \colon  Ax=Ax^*,x \ge 0\}$ is a singleton.
\end{proposition}
For perturbed matrices uniqueness is guaranteed for square reduced systems, and thus
less high sparsity values.
\begin{proposition}\cite[Th. 3.4]{Petra-Schnoerr-12a}\label{thm:SPCS1}  Let $A$ be the adjacency
matrix of a bipartite graph such that for all
subsets $X \subset C$ of $|X| \leq k$ left nodes, the set of neighbors $\calN(X)$ of $X$ satisfies
\begin{equation} \label{eq:Hassibi-condition}
|\calN(X)| \geq \delta 
\ell |\calN(\calN(X))\setminus \calN(\calN(X)^c)| 
\qquad\text{with}\qquad 
\delta > \frac{1}{\ell}.
\end{equation}
Then, for any $k$-sparse vector $x^*$, there exists a perturbation $\tilde A$ of $A$ such that the
solution set $\{x \colon \tilde Ax=\tilde Ax^*,x \ge 0\}$ is a singleton.
\end{proposition}
Recovery via perturbed underdetermined reduced systems is possible and our numerical results
from Section \ref{sec:num} suggest the following.
\begin{conjecture}\label{conj:SPCS2}  Let $A$ be the adjacency
matrix of a bipartite graph such that for all
subsets $X \subset C$ of $|X| \leq k$ left nodes, the set of neighbors $\calN(X)$ of $X$ satisfies
\begin{equation} \label{eq:inv-condition}
|\calN(X)| \geq \frac{1+\delta}{\ell} 
|\calN(\calN(X))\setminus \calN(\calN(X)^c)| 
\qquad\text{with}\qquad 
\delta >\frac{\sqrt{5}-1}{2}.
\end{equation}
Then, for any $\frac{k}{\ell}$-sparse vector $x^*$, there exists a perturbation $\tilde A$ of $A$ such that the
solution set $\{x \colon \tilde Ax=\tilde Ax^*,x \ge 0\}$ is a singleton.
\end{conjecture}

The consequences of Propositions \ref{thm:Wang}, \ref{thm:SPCS1} and Conjecture \ref{conj:SPCS2} are 
investigated in the following sections \ref{sec:hexagon_red_dim} and \ref{sec:cube_red_dim}
by working out critical values of the sparsity parameter $k$ for which the respective conditions are satisfied with high probability.

\section{3 Cameras - Left Degree equals 3}\label{sec:Hex3C}

In this section, we analyze the imaging set-up depicted in Figure \ref{fig:3_cam}, left panel, which also represents typical 3D scenarios encountered in practice with a coarse resolution only along the third coordinate, as shown by Figure \ref{fig:3_cam}, center panel.
\begin{figure}
\centerline{
\includegraphics[width=0.35\textwidth]{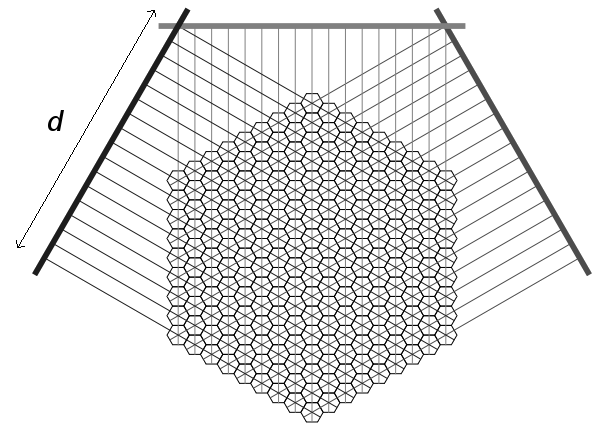}
\includegraphics[width=0.3\textwidth]{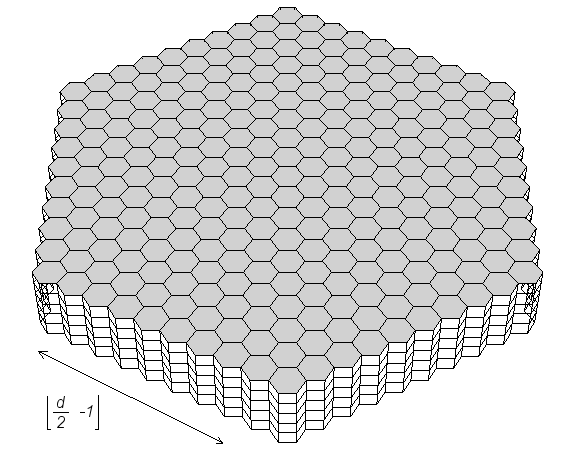}
\includegraphics[width=0.3\textwidth]{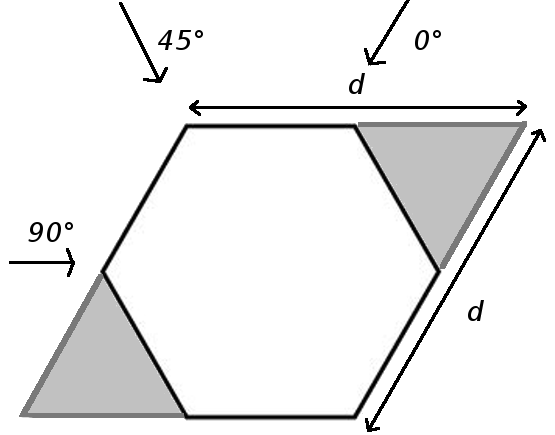}
}
\caption{Sketch of a 3-cameras setup in 2D. Left: The hexagonal area discretized in 
$3\frac{d^2+1}{4}$ equally sized cells projected on 3 1D cameras. 
The resulting projection matrix $A\in\{0, 1\}^{m\times n}$ is underdetermined,
with $m=3d$ and $n=3\frac{d^2+1}{4}$, where $\frac{d-1}{2}+1$ is the number of cells 
on each hexagon edge. Middle: This geometry can be easily extended to 3D by enhancing 
both cameras and volume by one dimension, thus representing scenarios of practical relevance 
when cameras are aligned on a line. Right: When considering a square area along with three 
projection directions (two orthogonal, one diagonal) one obtains a projection matrix with 
analogous reconstruction properties. The projection matrix \emph{equals} up to scaling the previous projection matrix corresponding to the hexagonal area if we remove the $2\cdot \frac{d^2-1}{8}$ cells in the marked corners along with incident rays.}
\label{fig:3_cam}
\end{figure}

\subsection{Imaging Geometry}\label{sec:hexagon_geom}

Cell centers $x_{c}$ of hexagonal cells $c \in C$ that partition a region of interest, are given by lattice points 
corresponding to integer linear combinations of two vectors $d^{i},\,i=1,2$,
\begin{gather}
 x_{c} = i_{1} d^{1} + i_{2} d^{2},\qquad
 d^{1} = \frac{1}{2} \bpm \sqrt{3} \\ 1 \epm,\quad
 d^{2} = \bpm 0 \\ 1 \epm,\quad
 (i_{1},i_{2}) \in \mc{I},
\end{gather}
for the index set
\begin{equation}
  \mc{I} = \big\{(i,j) \colon 
  -(d-1)/2 \leq i,j \leq (d-1)/2,\; |i+j| \leq (d-1)/2 
  \big\},
\end{equation}
with problem size $d \in \N$ that we assume (in this section) to be an odd number for simplicity.
The number $|R|$ of projections $r \in R = R_{1} \cup R_{2} \cup R_{3}$, where $R_{i},\,i=1,2,3$, corresponds to the rays of direction $i$, is
\begin{equation} \label{eq:def-NR}
 |R| = 3 |R_{i}| = 3 d.
\end{equation}
The number of cells incident with projection rays ranges over the interval
\begin{equation}
 \{(d+1)/2, (d+1)/2+1, \dotsc, d\}
\end{equation}
from the periphery towards the center. Thus, indexing with $r$ each projection ray along any particular direction $R_{i},\,i=1,2,3$, from one side of the hexagon across the center towards the opposite side, the numbers of cells incident with ray $r$ is
\begin{equation} \label{eq:def-Nc}
 |r| \in \{(d+1)/2,\dotsc, d,\dotsc, (d+1)/2\},\qquad
 r \in R_{i},\; i=1,2,3.
\end{equation}
The total number of cells is
\begin{equation} \label{eq:def-NC}
 |C| = \sum_{r \in R_{i}} |r|
 = 2 \sum_{j=(d+1)/2}^{d-1} j + d 
 = \frac{1}{4}(3 d^{2}+1),\quad i \in \{1,2,3\},
\end{equation}
and
\begin{equation} \label{eq:sum-NR-NC}
 \sum_{r \in R} |r| = 3 |C|.
\end{equation}
Accordingly, the system of equations representing the imaging geometry depicted by Figure \ref{fig:3_cam}, left panel, has dimensions
\begin{equation} \label{eq:def-Ab}
 A x = b,\qquad\qquad A \in \{0,1\}^{|R| \times |C|},\qquad
 b \in \R^{|R|}.
\end{equation}
Note that $|R| \ll |C|$.
For further reference, we define the quantities
\begin{subequations} \label{eq:def-pqr}
\begin{align}
 q_{r} &= \frac{|r|}{|C|}, &
 p_{r} &= 1-q_{r}, \\
 \ul{q}_{d} &= \min_{r \in R} q_{r}, &
 \ol{q}_{d} &= \max_{r \in R} q_{r}, \\
  \ul{p}_{d} &= 1-\ol{q}_{d}, &
 \ol{p}_{d} &= 1-\ul{q}_{d},
\end{align}
\end{subequations}
and list some further relations and approximations for large $d$,
\begin{subequations} \label{eq:basic-approximations}
\begin{align}
 \ul{q}_{d} &= \frac{2(d+1)}{3 d^{2}+1} 
 \approx \frac{2}{3 d}, &
 \ol{q}_{d} &= \frac{1}{3} \frac{|R|}{|C|} 
 \approx \frac{4}{3 d}, \\
 |R| \ul{q}_{d} &\approx 2, &
 |R| \ol{q}_{d} &\approx 4.
\end{align}
\end{subequations}

\subsection{Dimensions of Reduced Systems}
\label{sec:hexagon_red_dim}

We estimate the \emph{expected} dimensions \eqref{def:mn-red} of the reduced system \eqref{eq:red-system} based on
uniformly selecting $k$ cells at random locations. 

To each projection ray $r \in R$, we associate a binary random variable $X_{r}$ taking the value $X_{r}=1$ if \emph{not any} of the $k$ cells is incident with ray $r$, and $X_{r}=0$ otherwise. We call the event $X_{r}=1$ \emph{zero-measurement}.

We are interested in the random variable
\begin{equation} \label{eq:def-X}
 X = \sum_{r \in R} X_{r}
\end{equation}
that determines the number of projection rays not incident with any of the $k$ cells, that is the number of zero measurements. We set
\begin{equation} \label{eq:def-NR0}
 N_{R}^{0} := \EE[X],\qquad
 N_{R} := |R| - N_{R}^{0}.
\end{equation}
Hence, $N_{R}$ is the \emph{expected} size of the support $m_{red} = |\supp(b)|$ of the measurement vector $b$.
\begin{remark} \label{rem:Xr-dependence}
Note that random variables $X_{r}$ are \emph{not} independent because different projection rays may intersect. This dependency does not affect the expected value of $X$, but it does affect the deviation of observed values of $X$ from its expected value -- cf.~Section \ref{sec:hexagon-tailbound}.
\end{remark}
\begin{remark} \label{rem:ass-replace}
We do \emph{not} assume in the derivation below that $k$ \emph{different} cells are selected. 
In fact, a single cell may be occupied by more than a single particle in practice, because real particles are very small relative to the discretization cells $c$. The imaging optics enlarges the appearance of particles, and the action of physical projection rays is adequately represented by linear superposition.
\end{remark}
\begin{definition}[\textbf{Sparsity Parameter}] \label{def:k}
 We refer to the number $k$ introduced above as sparsity parameter. Thus, highly sparse scenarios correspond to low values $k$.
\end{definition}
\begin{lemma} \label{lem:NR0}
 The expected number $N_{R}^{0}$ of zero measurements is
\begin{equation} \label{eq:NR0}
 N_{R}^{0} = N_{R}^{0}(k) = \EE[X] = \sum_{r \in R} p_{r}^{k}.
\end{equation}
\end{lemma}
\begin{proof}
For $k=1$, $X_{r}$ has a Bernoulli distribution with
\begin{equation} \label{eq:def-pr}
 \EE[X_{r}] = \Pr[X_{r}=1] = 1 - \frac{|r|}{|C|}
 = 1-q_{r} = p_{r}.
\end{equation}
For $k$ independent trials, we have (cf.~Remark \ref{rem:ass-replace})
\begin{equation} \label{eq:exp-Xr}
 \EE[X_{r}] = \Pr[X_{r}=1] = p_{r}^{k}.
\end{equation}
By the linearity of expectations and \eqref{eq:exp-Xr}, we obtain 
\eqref{eq:NR0},
\begin{equation} \label{eq:exp-p0}
 N_{R}^{0} = \EE[X]
 = \sum_{r \in R} \EE[X_{r}] = 
 \sum_{r \in R} p_{r}^{k}.
\end{equation}
\end{proof}
We discuss few specific scenarios depending on the sparsity parameter $k$.
\begin{description}
\item[No particles] For $k=0$, we obviously have
\[ 
N_{R}^{0} = \sum_{r \in R} 1 = |R|.
\]
\item[High sparsity] By \eqref{eq:basic-approximations}, we have 
$\ul{q}_{d} \leq q_{r} \leq \ol{q}_{d}$, hence $q_{r} = \mc{O}(d^{-1})$. Thus, for large problem sizes $d$ and small values of $k$, 
\[
N_{R}^{0} \approx \sum_{r \in R} \Big(
\binom{k}{0} 1^{k} q_{r}^{0} - 
\binom{k}{1} 1^{k-1} q_{r}^{1} \Big) 
= \sum_{r \in R} (1-k q_{r}).
\]
By \eqref{eq:sum-NR-NC}, $\sum_{r \in R} q_{r} = 3$, hence 
\begin{equation} \label{eq:NR0-linear-lower-bound}
N_{R}^{0} \approx |R| - 3 k.
\end{equation}
This approximation says that for sufficiently small values of $k$ each randomly selected cell can be expected to create 3 independent measurements, which just reflects the fact that each cell is met by three projection rays.
\item[Less high sparsity] For increasing values of $k$ higher-order terms cannot longer be ignored, due to the increasing number of projection rays meeting \emph{several} particles. Taking the second-order term into account, we obtain in an analogous way
\begin{equation} \label{eq:NR0-quadratic-upper-bound}
\begin{aligned}
 N_{R}^{0} &\approx
 \sum_{r \in R} (1-k q_{r}+\frac{k (k-1)}{2} q_{r}^{2}) \\
 &\leq \sum_{r \in R} (1-k q_{r}+\frac{k (k-1)}{2} q_{r} \ol{q}_{d})
 = |R| - 3 k + \frac{3}{2} k (k-1) \ol{q}_{d},
\end{aligned}
\end{equation}
which is a fairly tight upper bound for values of $k$ and $N$ that are relevant to applications.
\end{description}

\vspace{0.5cm}
We consider next the \emph{expected} number of cells $n_{red} = |C_{b}|$ of cells supporting the set $R_{b}$ according to \eqref{eq:def-RbCb}. We denote this expected number by
\begin{equation}
 N_{C} := \EE[|C_{b}|],\qquad N_{C}^{0} := |C|-N_{C},
\end{equation}
and by $N_{C}^{0}$ the expected size of the complement.

Let $R = R_{1} \cup R_{2} \cup R_{3}$ denote the partition of all projection rays by the three directions. For each cell $c$, there are three unique rays $r_{i}(c) \in R_{i},\, i=1,2,3$, incident with $c$. Furthermore, for $i \neq j$ and some ray $r_{i} \in R_{i}$, let $R_{j}(r_{i})$ denote the set of rays that intersect with $r_{i}$. As before, $|r|$ denotes the number of cells covered by projection ray $r \in R$.
\begin{proposition} \label{prop:NC0}
For a given sparsity parameter $k$, the expected number 
of cells that can be recognized as empty based on the observations of random variables $\{X_{r}\}_{r \in R}$, is
\begin{subequations} \label{eq:NC}
\begin{align}
 N_{C}^{0} &= N_{C}^{0}(k) 
 = 3 N_{C}^{1} - 3 N_{C}^{2} + N_{C}^{3}, \label{eq:NC0} \\
 N_{C}^{1} &= \sum_{r \in R_{i}} |r| \left(1 - 
 \frac{|r|}{|C|} \right)^{k},\qquad
 \text{for any} \quad i \in \{1,2,3\}, 
 \label{eq:NC0-N1} \\
 N_{C}^{2} &= \sum_{r_{i} \in R_{i}} \sum_{r_{j} \in R_{j}(r_{i})}
 \left( 1 - \frac{|r_{i}| + |r_{j}|-1}{|C|} \right)^{k},
 \; \text{for any}\; i,j \in \{1,2,3\},\, i \neq j, 
 \label{eq:NC0-N2} \\ \label{eq:NC0-N3}
 N_{C}^{3} &= \sum_{c \in C} \left( 1 - \frac{\sum_{i=1}^{3} 
 |r_{i}(c)|-2}{|C|} \right)^{k}.
\end{align}
\end{subequations}
\end{proposition}
\begin{proof}
Each cell intersects with three projection rays $r_{i}(c),\,i=1,2,3$. Hence, given the rays corresponding to zero measurements, each cell that can be recognized as empty if either one, two or three rays from the set $\{r_{i}(c)\}_{i=1,2,3}$ belong to this set.

We therefore determine separately the expected number of removable cells (i) due to individual rays corresponding to zero measurements, (ii) due to all pairs of rays that intersect and correspond to zero measurements, and (iii) due to all triples of rays that intersect and correspond to zero measurements. 
The estimate \eqref{eq:NC0} then results from the inclusion-exclusion principle that combines these numbers so as to avoid overcounting, to obtain the desired estimate corresponding to the union of these events.

Consider each projection ray $r \in R_{i}$ for any fixed direction $i=1,2,3$. Because these rays do not intersect, the expected number of cells that can be removed based on the observation $\{X_{r}\}_{r \in R}$, is
\begin{equation}
 N_{C}^{1} = \EE\Big[\sum_{r \in R_{i}} X_{r} |r| \Big]
 = \sum_{r \in R_{i}} p_{r}^{k} |r|,
\end{equation}
by the linearity of expectations and \eqref{eq:exp-Xr}. Due to the symmetry of the setup, this number is the same for each direction $i=1,2,3$. Hence we multiply $N_{C}^{1}$ by $3$ in \eqref{eq:NC0}.

Consider next pairs of directions $i,j \in \{1,2,3\},\, i \neq j$. For $i$ fixed, the expected number of empty cells based on a zero measurement 
corresponding to some ray $r_{i} \in R_{i}$ \emph{and all} rays $r_{j} \in R_{j}(r_{i})$ intersecting with $r_{i}$, is
\begin{equation}
 N_{C}^{2} 
 = \EE\Big[ \sum_{r_{i} \in R_{i}} \sum_{r_{j} \in R_{j}(r_{i})} 
 X_{r_{i}} X_{r_{j}} \Big].
\end{equation}
The linearity of expectations and $\EE[X_{r_{i}} X_{r_{j}}] = \Pr\big[(X_{r_{i}}=1) \wedge (X_{r_{j}}=1)\big]$ gives \eqref{eq:NC0-N2}. Due to symmetry, we have to multiply $N_{C}^{2}$ by $3$ in \eqref{eq:NC0}.

Finally, the expected number of empty cells that correspond to observed zero measurements along \emph{all three} projection directions, is
\begin{equation}
 N_{C}^{3} = \EE\Big[\sum_{c \in C} 
 \prod_{i=1}^{3} X_{r_{i}(c)} \Big],
\end{equation}
which equals \eqref{eq:NC0-N3}.
\end{proof}

An immediate consequence of Lemma \ref{lem:NR0} and Prop.~\ref{prop:NC0} is
\begin{corollary} \label{cor:Ared}
For a given value of the sparsity parameter $k$, the expected dimensions of the reduced system \eqref{eq:red-system} are
\begin{equation} \label{eq:Ared}
 m_{red}=N_{R}-N_{R}^{0},\quad
 n_{red}=N_{C}-N_{C}^{0},
\end{equation}
with $N_{R}^{0}, N_{C}^{0}$ given by \eqref{eq:NR0} and \eqref{eq:NC}.
\end{corollary}
%

\subsection{A Tail Bound}
\label{sec:hexagon-tailbound}

We are interested in how sharply the random number $X$ of zero measurements peaks around its expected value $N_{R}^{0} = \EE[X]$ given by \eqref{eq:NR0}.
 
Because the random variables $X_{r},\,r \in R$, are \emph{not} independent due to the intersection of projection rays, we apply the following classical inequality for bounding the deviation of a random variable from its expected value based on martingales, that is on sequences of random variables $(X_{i})$ defined on a finite probability space $(\Omega, \mc{F}, \mu)$ satisfying
\begin{equation} \label{eq:condition-martingale}
 \EE[X_{i+1}|\mc{F}_{i}] = X_{i},\qquad
 \text{for all}\quad i \geq 1,
\end{equation}
where $\mc{F}_{i}$ denotes an increasing sequence of $\sigma$-fields in $\mc{F}$ with $X_{i}$ being $\mc{F}_{i}$-measurable.
\begin{theorem}[Azuma's Inequality \cite{Azuma1967,DasGupta2008}]
\label{thm:Azuma}
 Let $(X_{i})_{i=0,1,2,\dotsc}$ be a sequence of random variables such that for each $i$,
\begin{equation} \label{eq:def-ci}
 |X_{i}-X_{i-1}| \leq c_{i}.
\end{equation}
Then, for all $j \geq 0$ and any $\delta > 0$,
\begin{equation}
 \Pr\big(|X_{j}-X_{0}| \geq \delta\big) \leq 
 2 \exp\Big( -\frac{\delta^{2}}{2 \sum_{i=1}^{j} c_{i}^{2}}\Big).
\end{equation}
\end{theorem}
Let $\mc{F}_{i} \subset 2^{R},\, i=0,1,2,\dotsc$, denote the $\sigma$-field generated by the collection of subsets of $R$ that correspond to all possible events after having observed $i$ randomly selected cells. We set $\mc{F}_{0} = \{\emptyset,R\}$. Because observing cell $i+1$ just further partitions the current state based on the previously observed $i$ cells by possibly removing some ray (or rays) from the set of zero measurements, we have a nested sequence (filtration) $\mc{F}_{0} \subseteq \mc{F}_{1} \subseteq \dotsb \subseteq \mc{F}_{k}$ of the set $2^{R}$ of all subsets of $R$.

Based on this, for a fixed value of the sparsity parameter $k$, we define the sequence of random variables
\begin{equation}
 Y_{i} = \EE[X|\mc{F}_{i}],\quad i=0,1,\dotsc,k,
\end{equation}
where $Y_{i},\,i=0,1,\dotsc,k-1$, are the random variable specifying the expected number of zero measurements after having observed $k$ randomly selected cells, conditioned on the subset of events $\mc{F}_{i}$ determined by the observation of $i$ randomly selected cells. Consequently, $Y_{0}=\EE[X]=N_{R}^{0}$ due to the absence of any information, and $Y_{k} = X$ is just the observed number of zero measurements. 
The sequence $(Y_{i})_{i=0,\dotsc,k}$ is a martingale by construction satisfying $\EE[Y_{i+1}|\mc{F}_{i}]=Y_{i}$, that is condition \eqref{eq:condition-martingale}.
\begin{proposition} \label{prop:NR0-deviation}
 Let $N_{R}^{0}=\EE[X]$ be the expected number of zero measurements for a given sparsity parameter $k$, given by \eqref{eq:NR0}. Then, for any $\delta > 0$, 
\begin{equation} \label{eq:NR0-deviation}
 \Pr\big[|X-N_{R}^{0}| \geq \delta\big] \;\leq\;
 2 \exp\bigg(
 -\frac{1-\ol{p}_{d}^{2}}{18 (1-\ol{p}_{d}^{2k})}
 \;\delta^{2}
 \bigg).
\end{equation}
\end{proposition}
\begin{proof}
 Let $R^{0}_{i-1} \subset R$ denote the subset of rays with zero measurements after the random selection of $i-1 < k$ cells. For the remaining $k-(i-1)$ trials, the probability that not any cell incident with some ray $r \in R^{0}_{i-1}$ will be selected, is 
\begin{equation}
 p_{r}^{k-(i-1)} = \EE[X_{r}|\mc{F}_{i-1}],
\end{equation}
with $p_{r}$ given by \eqref{eq:def-pr}. Consequently, by the linearity of expectations, the expectation $Y_{i-1}$ of zero measurements, given the number $|R^{0}_{i-1}|$ of zero measurements after the selection of $i-1$ cells, is
\begin{equation}
 Y_{i-1} = \EE[X|\mc{F}_{i-1}]
 = \sum_{r \in R^{0}_{i-1}} p_{r}^{k-(i-1)}.
\end{equation}
Now suppose we observe the random selection of the $i$-th cell. We distinguish two possible cases.
\begin{enumerate}
\item
 Cell $i$ is not incident with any ray $r \in R^{0}_{i-1}$. Then the number of zero measurements remains the same, and
\begin{equation}
 Y_{i} = \sum_{r \in R^{0}_{i-1}} p_{r}^{k-i}.
\end{equation}
Furthermore, 
\begin{equation} \label{eq:Azuma-estimate-2}
\begin{aligned}
Y_{i}-Y_{i-1} &= \sum_{r \in R^{0}_{i-1}} \big(
 p_{r}^{k-i} - p_{r}^{k-(i-1)} \big)
 = \sum_{r \in R^{0}_{i-1}} p_{r}^{k-i} (1-p_{r}) \\
 &\leq  \ol{p}_{d}^{k-i} \sum_{r \in R} q_{r}
 = 3 \ol{p}_{d}^{k-i}.
\end{aligned}
\end{equation}
\item
Cell $i$ is incident with one, two or three rays contained in $R^{0}_{i-1}$. Let $R^{0}_{i}$ denote the set $R^{0}_{i-1}$ after removing these rays. Then
\[
 Y_{i} = \sum_{r \in R^{0}_{i}} p_{r}^{k-i}.
\]
Furthermore, since $R^{0}_{i} \subset R^{0}_{i-1}$ and
$|R^{0}_{i-1} \setminus R^{0}_{i}| \leq 3$, 
\begin{align*}
Y_{i-1} - Y_{i} &=
\sum_{r \in R^{0}_{i-1} \setminus R^{0}_{i}} p_{r}^{k-(i-1)}
- \sum_{r \in R^{0}_{i}} \big(
p_{r}^{k-i} - p_{r}^{k-(i-1)} \big) \\
&\leq 3 \ol{p}_{d}^{k-i+1} 
- \sum_{r \in R^{0}_{i}} 
\ol{p}_{d}^{k-i} \ul{q}_{d}.
\end{align*}
Further upper bounding by dropping the second sum shows that the resulting first term is still smaller than the bound \eqref{eq:Azuma-estimate-2}.
\end{enumerate}
As a result, we consider the larger bound \eqref{eq:Azuma-estimate-2} of these two cases and compute
\[
\sum_{i=1}^{k} \big(3 \ol{p}_{d}^{k-i}\big)^{2}
= 9 \frac{1-\ol{p}_{d}^{2k}}{1-\ol{p}_{d}^{2}}.
\]
Applying Theorem \ref{thm:Azuma} completes the proof.
\end{proof}
\begin{remark} \label{rem:choice-of-k}
Expanding the r.h.s.~of \eqref{eq:NR0-deviation} around $0$ in terms of the variable $d^{-1}$ shows
\begin{equation}
 \Pr\big[|X-N_{R}^{0}| \geq \delta\big] \;\leq\;
 2 \exp\Big(-\frac{\delta^{2}}{18 k} \Big) 
 \qquad\text{for}\qquad
 d \to \infty.
\end{equation}
This indicates appropriate choices $k = k(d)$ for large but finite problem sizes $d$ occurring in applications, so as to bound the deviation of $N_{R}^{0}$ from its expected value. As a result, for such choices of $k$, our analysis based on expected values of the key system parameters will hold in applications with high probability.
\end{remark}

\subsection{Critical Sparsity Values and Recovery}

We derived the expected number $N_R(k)$ of nonzero measurements $m_{red}$ \eqref{def:mn-red} induced by random $k$-sparse
vectors $x\in\R^n_{k,+}$ and the corresponding expected number $N_R(k)$ of non redundant cells $n_{red}$.
The tail bound, Prop. \ref{prop:NR0-deviation}, guarantees that the dimensions of reduced systems concentrate around the
derived expected values, explaining the threshold effects of unique recovery from sparse tomographic measurements.
 
We now introduce some further notations and discuss the implication of Section \ref{sec:uniqueness} on exact recovery of $x\in\Rn_{k,+}$.
Let $N_R(k)$ and $N_C(k)$ be the expected dimensions of the reduced system induced by
a random $k$-sparse nonnegative vector as detailed in Corollary \ref{cor:Ared}.
Let $\delta=\frac{\sqrt{5}-1}{2}$ and denote by $k_\delta, k_{crit}$ and $k_{1/\delta}$ the on $d$ dependent sparsity values which solve
the equations
\begin{align}
   N_R(k_\delta) & =\delta \ell N_C(k_\delta), \label{eq:kdelta-criterion}\\
   N_R(k_{crit}) & = N_C(k_{crit}),  \label{eq:kcrit-criterion}\\
   N_R(k_{opt}) & =\frac{1}{2} N_C(k_{opt}), \label{eq:kopt-criterion}\\
   N_R(k_{1/\delta})& = \frac{1+\delta}{\ell}N_C(k_{1/\delta}). \label{eq:kinv-criterion}
\end{align} 

In what follows, the phrase \emph{with high probability} refers to values of the sparsity parameter $k$ for which random supports $|\supp(b)|$ concentrate around the crucial expected value $N_{R}$ according to Prop.~\ref{prop:NR0-deviation}, thus yielding a desired threshold effect.

\begin{proposition} \label{prop:appl-Wang}
 The system $A x = b$, with measurement matrix $A$, admits unique recovery of $k$-sparse non-negative vectors $x$ with high probability, if
\begin{equation}\label{tilde-kdelta-criterion}
 k \leq \frac{N_{C}(k_{\delta})}{1+\delta}=:\tilde k_\delta\ .
\end{equation}
\end{proposition}

For perturbed systems we have.
\begin{proposition} \label{prop:appl-Hassibi}
 The system $\tilde A x = b$, with perturbed measurement matrix $\tilde A$, admits unique recovery of $k$-sparse non-negative vectors $x$ with high probability, if $k$ satisfies condition $k \leq k_{crit}$.
\end{proposition}

In case Conjecture \ref{conj:SPCS2} holds, uniqueness of $x\in\Rn_{k,+}$ would be guaranteed if $k\le k_{1/\delta}$.
Finally, we comment on the maximal sparsity threshold $k_{opt}$, in case reduced systems would follow a symmetric distribution with respect to 
the origin and columns would be in general position.

\begin{figure}
\centerline{
\includegraphics[width=0.3\textwidth]{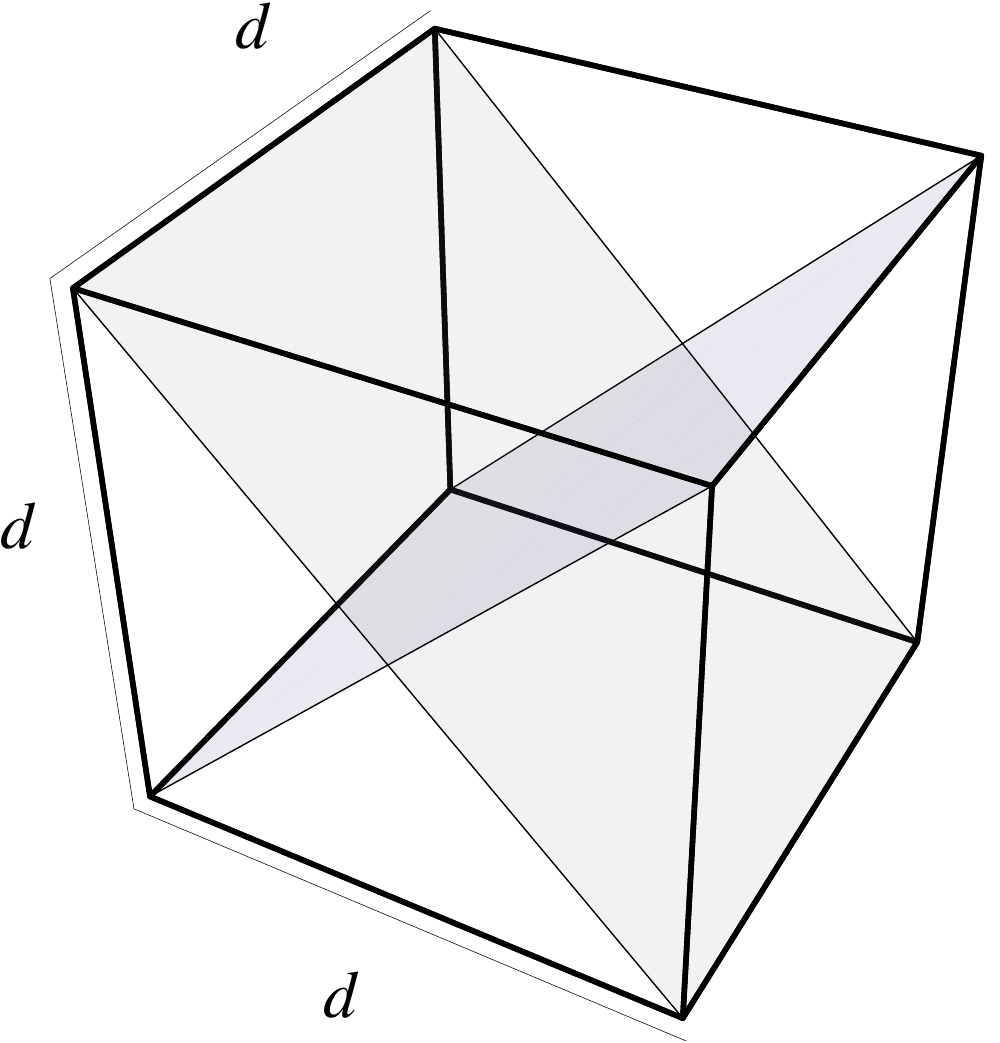}
\hspace{0.025\textwidth}
\includegraphics[width=0.3\textwidth]{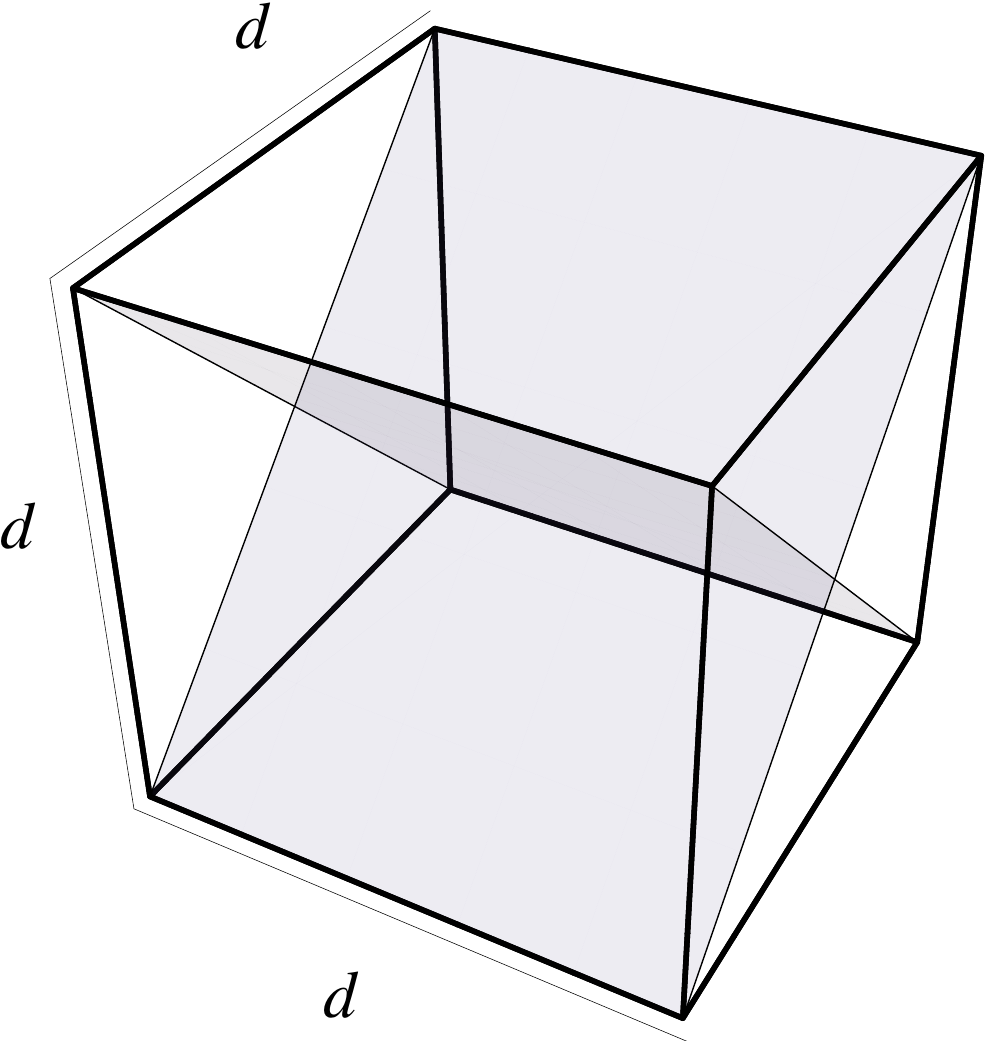}
\hspace{0.025\textwidth}
\includegraphics[width=0.25\textwidth]{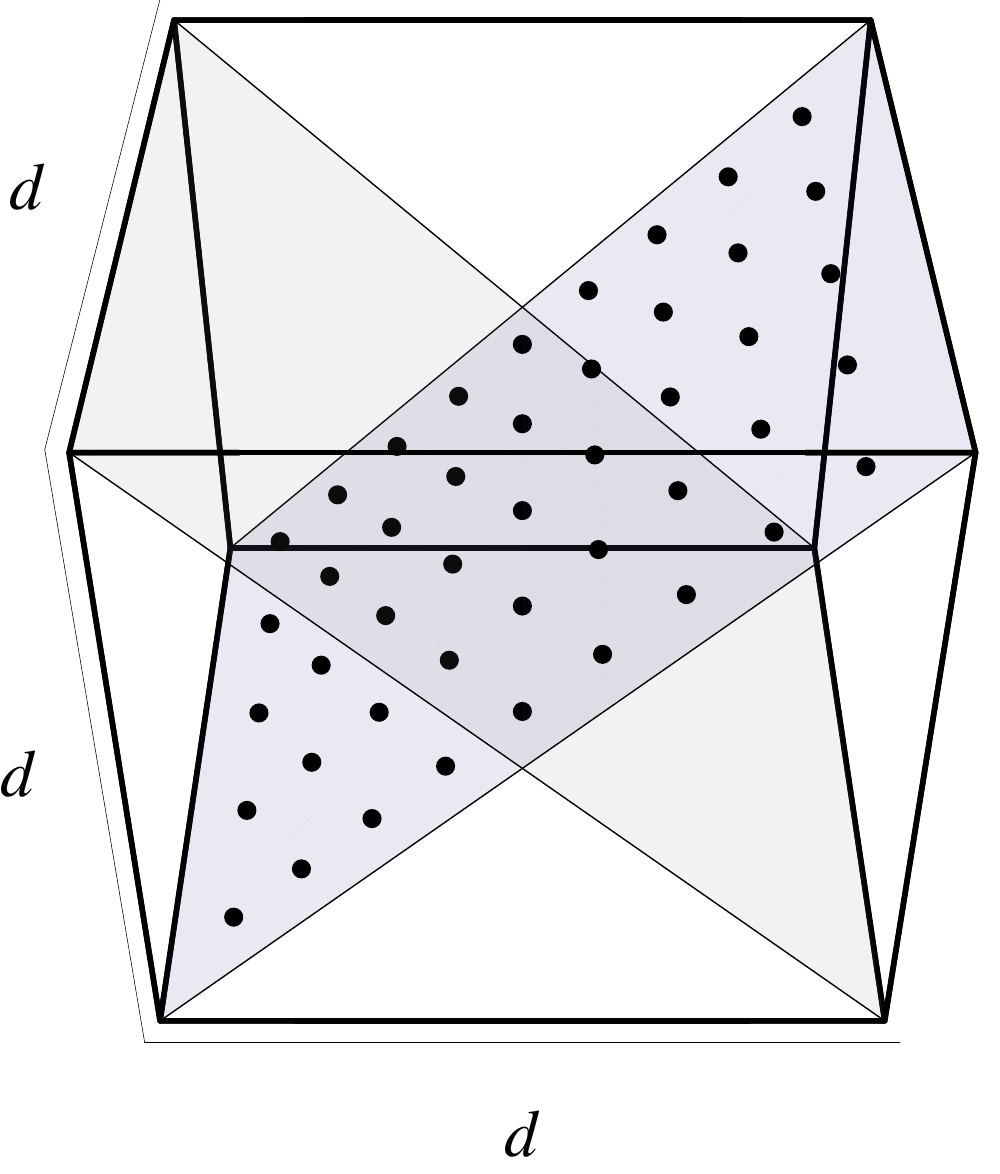}
}
\caption{
Imaging setup with 4 cameras corresponding to the image planes, shown as two pairs in the left and center panel, respectively. Right panel: Cell centers projected onto the first image plane are shown as dots for the case $d=5$. The cube $\Omega = [0,d]^{3}$ is discretized into $d^{3}$ cells and projected along $4 \cdot d (2 d-1)$ rays.
}
\label{fig:4C3D}
\end{figure}
\section{4 Cameras - Left Degree equals 4}\label{sec:4C3D}
We consider the imaging set-up depicted by Figure 
\ref{fig:4C3D} and conduct a probabilistic analysis of its recovery properties, analogous to \ref{sec:hexagon_red_dim}. This scenario 
is straightforward to realize and should also be particularly relevant to practical applications.

\subsection{Imaging Geometry}

Each coordinate of the unit cube $\Omega = [0,d]^{3}$ is discretized into the intervals $\{0,1,2,\dotsc,d\}$, resulting in $d^{3}$ voxels with coordinates
\begin{equation} \label{eq:def-C-cells}
C = \big\{c = (i,j,l)-\frac{1}{2}(1,1,1) \colon i,j,l \in [d]\big\}.
\end{equation}
There are 4 sets of parallel projection rays corresponding to the normals of the image planes depicted in Fig.~\ref{fig:4C3D},
\begin{equation} \label{eq:4-normals}
n^{1} = \frac{1}{\sqrt{2}} (-1,0,1),\quad
n^{2} = \frac{1}{\sqrt{2}} (1,0,1),\quad
n^{3} = \frac{1}{\sqrt{2}} (0,-1,1),\quad
n^{4} = \frac{1}{\sqrt{2}} (0,1,1).
\end{equation}
We denote the set of projection rays and its partition corresponding to the 4 directions 
\begin{equation} \label{eq:R-1-4}
R = \cup_{l=1}^{4} R_{l}.
\end{equation}
Each set $R_{i}$ contains $(2 d-1) \cdot d$ projection rays whose measurements yields a projection image with $(2 d-1) \times d$ pixels. We index and denote the pixels by $(s,t)$, and the projection rays through these pixels by
\[
r^{i}_{s,t} \in R_{i},\quad i\in\{1,2,3,4\}.
\] 
For each cell $c \in C$ indexed by $i,j,l \in [d]$ according to \eqref{eq:def-C-cells}, we represent the corresponding pixels after a suitable transformation by
\begin{subequations} \label{eq:voxel-projections}
 \begin{align}
 (s_{1},t_{1}) &= (i+l-1-d,j), \qquad&\qquad
 s_{1} &\in [1-d,d-1],\; t_{1} \in [d], \\
 (s_{2},t_{2}) &= (i-l,j), \qquad&\qquad
 s_{2} &\in [1-d,d-1],\; t_{2} \in [d], \\
 (s_{3},t_{3}) &= (i,j+l-1-d), \qquad&\qquad
 s_{3} &\in [d],\; t_{3} \in [1-d,d-1], \\
 (s_{4},t_{4}) &= (i,j-l), \qquad&\qquad
 s_{4} &\in [d],\; t_{4} \in [1-d,d-1].
 \end{align}
\end{subequations}
The cardinalities of the projection rays, i.e.~the number of cells covered by each projection ray, are
\begin{subequations} \label{eq:r1234-cells}
\begin{align}
a \in \{1,2\}\colon\qquad
|r^{a}_{s,t}| &= d-|s|,
\qquad s \in [1-d,d-1],\quad
t \in [d], 
\label{eq:r12-cells} \\ \label{eq:r34-cells}
b \in \{3,4\}\colon\qquad
|r^{b}_{s,t}| &= d-|t|,
\qquad s \in [d],\quad
t \in [1-d,d-1].
\end{align}
\end{subequations}
We observe the symmetries
\begin{equation} \label{eq:r1234-symmetry}
 |r^{a}_{-s,t}| = |r^{a}_{s,t}|,\qquad
 |r^{a}_{t,s}| = |r^{b}_{s,t}|
\end{equation}
and define 
\begin{equation} \label{eq:def-rs}
 |r_{s}| := |r^{a}_{s,1}|
\end{equation}
because $|r^{a}_{s,t}|$ does not vary with $t$.
Summing up the cells covered by all rays along the first direction, for example, we obtain by \eqref{eq:r12-cells}, \eqref{eq:r1234-symmetry} and \eqref{eq:def-rs},
\begin{align*}
\sum_{r^{1} \in R_{1}} |r^{1}|
&= \sum_{t \in [d]} \sum_{s=1-d}^{d-1} |r^{1}_{s,t}| = 
d \sum_{s=1-d}^{d-1}|r_{s}| = 
d \big(d + 2 \sum_{s=1}^{d-1} (d-s)\big) = d^{3} = |C|.
\end{align*}
We set
\begin{equation} \label{eq:def-Rk1k2}
\begin{aligned}
 R(k_{1},k_{2}) &:= 
 \Big(1-\frac{|r_{k_{1}}| + |r_{k_{2}}|-1}{d^{3}}\Big)^{k}, \\
 R(k_{1},k_{2},k_{3}) &:= 
 \Big(1-\frac{|r_{k_{1}}| + |r_{k_{2}}| + |r_{k_{3}}|
 -2}{d^{3}}\Big)^{k}, \\
 R(k_{1},k_{2},k_{3},k_{4}) &:= 
 \Big(1-\frac{|r_{k_{1}}| + |r_{k_{2}}| + |r_{k_{3}}| + |r_{k_{4}}|
 -3}{d^{3}}\Big)^{k}.
\end{aligned}
\end{equation}

We conduct next for this setup the analysis analogous to Section \ref{sec:hexagon_red_dim}, in order to compute the expected 
size of the reduced system \eqref{def:mn-red} for random $k$-sparse vectors $x$.

\subsection{Dimensions of Reduced Systems}
\label{sec:cube_red_dim}

We first compute the expected number of measurements $m_{red}$ \eqref{def:mn-red} as a function of the sparsity parameter $k$.
\begin{lemma}\label{lem:NR-cube}
The expected number $m_{red} = N_{R}$ of non-zero measurements is
\begin{subequations}
\begin{align}
N_{R} &= N_{R}(k) = \EE[|\supp(b)|] = |R|-N_{R}^{0}
= 4 d (2 d-1) - N_{R}^{0}, \\
N_{R}^{0} &= 4 d \bigg( \Big(1-\frac{1}{d^{2}}\Big)^{k} + 2 
\sum_{s=1}^{d-1} \Big(1-\frac{s}{d^{3}}\Big)^{k} \bigg).
\end{align}
\end{subequations}
\end{lemma}
\begin{proof}
Taking into account symmetry, we have
\begin{align*}
 N_{R}^{0} &= \EE\Big[\sum_{r \in R} X_{r}\Big]
 = \sum_{r \in R} p_{r}^{k}
 = 4 \sum_{r^{1} \in R_{1}} \Big(1-\frac{|r^{1}|}{|C|}\Big)^{k}.
\end{align*}
Applying \eqref{eq:r12-cells} yields the assertion.
\end{proof}
\begin{proposition}\label{prop:NC-cube}
 The expected size $n_{red} = N_{C}$ of subset of cells that support random subsets $R_{b} \subset \R$ of observed non-zero measurements, is 
\begin{equation}
 N_{C} = N_{C}(k) = 
 d^{3} - N_{C}^{1} + N_{C}^{2} - N_{C}^{3} + N_{C}^{4}
\end{equation}
where
\begin{equation}
\begin{aligned}
 N_{C}^{1} &=  4 d \bigg( d \Big(1-\frac{1}{d^{2}}\Big)^{k} + 
  2 \sum_{s=1}^{d-1} s \Big(1-\frac{s}{d^{3}}\Big)^{k} \bigg), \\
 N_{C}^{2} &= 2 d \sum_{i,l \in [d]} R(l+i-1-d,i-l) 
 + 4 \sum_{i,j,l\in [d]} R(l-i,l-j), \\
 N_{C}^{3} &= 2 \sum_{i,j,l\in [d]}\big(
 R(l+i-1-d,l-i,l-j) + R(l-i,l-j,l+j-1-d) \big), \\
 N_{C}^{4} &= \sum_{i,j,l \in [n]} R(l+i-1-d,l-i,l+j-1-d,l-j),
\end{aligned}
\end{equation}
end the functions $R$ are given by \eqref{eq:def-Rk1k2}.
\end{proposition}
\begin{proof}
We consider for each cell $c \in C$ the quadruple of projection rays $(r^{1}_{C},r^{2}_{C},r^{3}_{C},r^{4}_{C})$ meeting in this cell, and the corresponding partition \eqref{eq:R-1-4} of projection rays. Cell $c$ is contained in the set $C_{b}$ \eqref{eq:def-RbCb} supporting $R_{b}$ if not any ray of the corresponding quadruple returns a zero measurement. Thus,
\begin{equation} \label{eq:NC-all-proof}
\begin{aligned}
 N_{C} &= \EE\Big[\sum_{c \in C}
 (1-X_{r^{1}_{c}}) (1-X_{r^{2}_{c}}) 
 (1-X_{r^{3}_{c}}) (1-X_{r^{4}_{c}})
 \Big] \\
 &= \sum_{c \in C} \Big(1 - \sum_{i=1}^{4} \EE[X_{r^{i}_{c}}]
 + \sum_{1 \leq i \leq j \leq 4} \EE[X_{r^{i}_{c}} X_{r^{j}_{c}}]
 - \sum_{1 \leq i \leq j \leq l \leq 4} 
 \EE[X_{r^{i}_{c}} X_{r^{j}_{c}} X_{r^{l}_{c}}]
 + \EE[X_{r^{1}_{c}} X_{r^{2}_{c}} X_{r^{3}_{c}} X_{r^{4}_{c}}]
 \Big) \\
 &= \sum_{c \in C} \bigg(1 
 - \sum_{i=1}^{4} \Big(1-\frac{|r^{i}_{c}|}{d^{3}}\Big)^{k}  
 + \sum_{1 \leq i \leq j \leq 4} 
 \Big(1-\frac{|r^{i}_{c} \cup r^{j}_{c}|}{d^{3}}\Big)^{k} \\
 &\qquad
 - \sum_{1 \leq i \leq j \leq l \leq 4} 
 \Big(1-\frac{|r^{i}_{c} \cup r^{j}_{c} \cup r^{l}_{c}|}{d^{3}} \Big)^{k}
 + \Big(1-\frac{|\cup_{i=1}^{4} r^{i}_{c}|}{d^{3}} \Big)^{k}
 \bigg)
\end{aligned}
\end{equation}
We consider each term in turn. 
\begin{enumerate}[(i)]
\item
As for the first term, we obviously have $|C|=d^{3}$. 
\item
Concerning the second term,  
taking symmetry into account we compute, 
\begin{align*}
  \sum_{c \in C} \sum_{i=1}^{4} 
  \Big(1-\frac{|r^{i}_{c}|}{d^{3}}\Big)^{k}
  &= 4 \sum_{c \in C} \Big(1-\frac{|r^{1}_{c}|}{d^{3}}\Big)^{k}
  = 4 \sum_{r^{1} \in R_{1}} \sum_{c \in r^{1}} 
  \Big(1-\frac{|r^{1}_{c}|}{d^{3}}\Big)^{k}.
\end{align*}
Since $r^{1}_{c} = r^{1}$ for all $c \in r^{1}$, we obtain using \eqref{eq:r1234-cells},
\begin{align*}
\sum_{c \in C} \sum_{i=1}^{4} \EE[X_{r^{i}_{c}}]
  &= 4 \sum_{r^{1} \in R_{1}} 
  |r^{1}| \Big(1-\frac{|r^{1}|}{d^{3}}\Big)^{k} 
  = 4 d \sum_{s_{1} = 1-d}^{d-1} 
  |r^{1}_{s_{1},t_{1}}| 
  \Big(1-\frac{|r^{1}_{s_{1},t_{1}}|}{d^{3}}\Big)^{k} \\
  &= 4 d \bigg( d \Big(1-\frac{1}{d^{2}}\Big)^{k} + 
  2 \sum_{s=1}^{d-1} s \Big(1-\frac{s}{d^{3}}\Big)^{k} \bigg).
\end{align*}
\item
Concerning the third term of \eqref{eq:NC-all-proof}, we consider first the contribution of the pair of directions $(i,j) = (1,2)$. 
Replacing cell-indices of projection rays by pixel indices according to \eqref{eq:voxel-projections}, we have using \eqref{eq:r1234-cells} and \eqref{eq:def-rs},
\begin{equation} \label{eq:NC-r12-term}
\begin{aligned}
 \sum_{c \in C} 
 \Big(1-\frac{|r^{1}_{c} \cup r^{2}_{c}|}{d^{3}}\Big)^{k}
 &= \sum_{i,j,l \in [d]} 
 \Big(1-\frac{|r^{1}_{(i+l-1-d,j)}| + |r^{2}_{(i-l,j)}| 
 - 1}{d^{3}}\Big)^{k} \\
 &= d \sum_{i,l \in [d]} 
 \Big(1-\frac{|r_{i+l-1-d}| + |r_{i-l}| 
 - 1}{d^{3}}\Big)^{k} \\
 &= d \sum_{i,l \in [d]} R(i+l-1-d,i-l),
\end{aligned}
\end{equation}
where the factor $d$ appears because the summand does not depend on $j$ by \eqref{eq:r12-cells}, and $R$ is defined by \eqref{eq:def-Rk1k2}. For the pair of directions $(3,4)$, we get
\begin{equation} \label{eq:NC-r34-term}
\begin{aligned}
 \sum_{c \in C} 
 \Big(1-\frac{|r^{3}_{c} \cup r^{4}_{c}|}{d^{3}}\Big)^{k}
 &= \sum_{i,j,l \in [d]} 
 \Big(1-\frac{|r^{3}_{(i,j+l-1-d)}| + |r^{4}_{(i,j-l)}| 
 - 1}{d^{3}}\Big)^{k},
\end{aligned}
\end{equation}
which equals \eqref{eq:NC-r12-term} due to the symmetry \eqref{eq:r1234-symmetry}.

Next, we consider the pair of directions $(1,3)$. Taking into account the symmetry \eqref{eq:r1234-symmetry} and using \eqref{eq:def-Rk1k2},  we obtain
\begin{equation}
\begin{aligned}
 &\sum_{c \in C} 
 \Big(1-\frac{|r^{1}_{c} \cup r^{3}_{c}|}{d^{3}}\Big)^{k}
 = \sum_{i,j,l \in [d]} 
 \Big(1-\frac{|r^{1}_{(i+l-1-d,j)}| + |r^{3}_{(i,j+l-1-d)}| 
 - 1}{d^{3}}\Big)^{k} \\
 &\quad= \sum_{i,j,l \in [d]} 
 \Big(1-\frac{|r_{i+l-1-d}| + |r_{j+l-1-d}| 
 - 1}{d^{3}}\Big)^{k} 
 = \sum_{i,j,l \in [d]} 
 \Big(1-\frac{|r_{l-i}| + |r_{l-j}| - 1}{d^{3}}\Big)^{k} 
 \end{aligned}
\end{equation}
In the same way it can be shown that the remaining pairs of directions
$(1,4), (2,3), (2,4)$ each contributes the last expression.
\item
Concerning the fourth term of (4.10), we get for the triple of directions $(1,2,3)$ the contribution
\begin{equation}
\begin{aligned}
 \sum_{c \in C} \Big(1-\frac{|r^{1}_{c} \cup r^{2}_{c} \cup r^{3}_{c}|}{d^{3}} \Big)^{k}
 &= \sum_{i,j,l \in [n]} 
 \Big(1-\frac{|r_{i+l-1-d}| + |r_{i-l}|
 + |r_{j+l-1-d}| - 2}{d^{3}}\Big)^{k} \\
 &= \sum_{i,j,l \in [n]} 
 R(i+l-1-d,l-i,j+l-1-d)\big),
\end{aligned}
\end{equation}
and likewise for the remaining triples
\begin{equation}
\begin{aligned}
 (1,2,4)\colon\qquad &
 \sum_{i,j,l \in [n]} R(i+l-1-d,l-i,l-j), \\
 (1,3,4)\colon\qquad &
 \sum_{i,j,l \in [n]} R(i+l-1-d,j+l-1-d,l-j), \\
 (2,3,4)\colon\qquad &
 \sum_{i,j,l \in [n]} R(l-i,j+l-1-d,l-j).
\end{aligned}
\end{equation}
Evidently, the first and last pair of expressions are equal, respectively.
\item
Finally, the expression for the last term of \eqref{eq:NC-all-proof} is immediate.
\end{enumerate}
\end{proof}

We conclude this section by stressing that critical sparsity values
$k_\delta$ \eqref{eq:kdelta-criterion}, $k_{crit}$ \eqref{eq:kcrit-criterion},
$k_{opt}$ \eqref{eq:kopt-criterion}, $k_{1/\delta})$ \eqref{eq:kinv-criterion}, can be worked out
based on the just derived values $N_R$ and $N_C$. 
A tail bound may be derived analogously to Prop. 3.5. We omit this redundant detail due to space constraints.

\begin{figure}
\begin{center}
\begin{tabular}{cc}
\includegraphics[clip,width=0.45\textwidth]{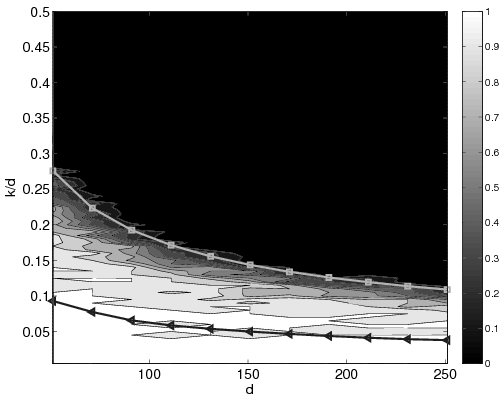} &
\includegraphics[clip,width=0.45\textwidth]{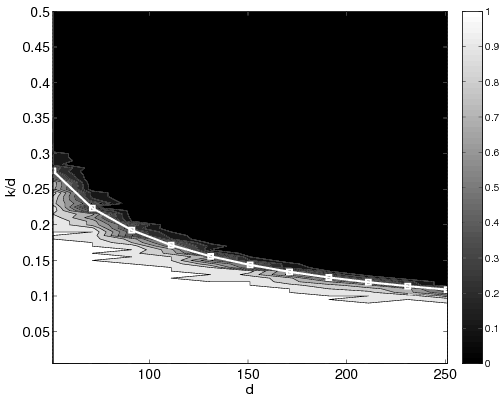}\\
\includegraphics[clip,width=0.45\textwidth]{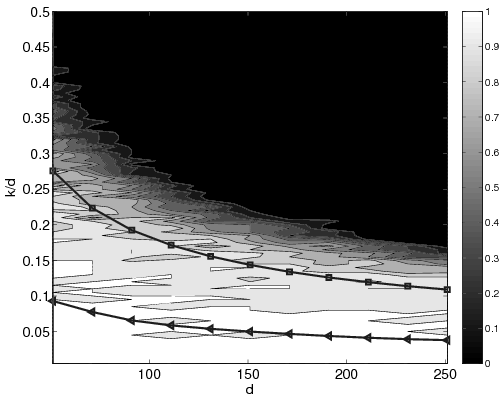}&
\includegraphics[clip,width=0.45\textwidth]{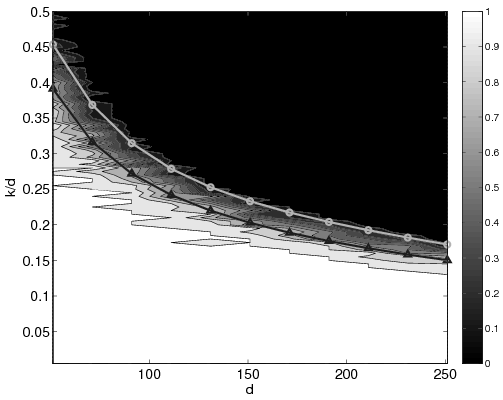}
\end{tabular}
\end{center}
\caption{
Success and failure empirical phase transitions for the 2D, 3-cameras case, hexagonal area, from Section \ref{sec:Hex3C}, Fig. \ref{fig:3_cam}, left.
Reduced unperturbed (top left) and perturbed (top right) matrices are overdetermined
and of full rank with hight probability if the corresponding sparsity level  is below $k_{\delta}$ (unperturbed case) or
$k_{crit}$ (perturbed case). The $\triangleleft$-marked curve depicts $\tilde k_{\delta}/d^2$ \eqref{eq:kdelta-criterion}, 
and the $\square$-marked curve, $k_{crit}/d^2$ from \eqref{eq:kdelta-criterion}.
Probability of uniqueness in $[0,1]^n$ of a $k=\rho d^2$ sparse binary vector
for unperturbed (bottom left) and  perturbed matrices (bottom right).
This probability is high below  $\tilde k_{\delta}/d^2$ and decreases slowly, for the unperturbed case (bottom left).
In the perturbed case the empirical probability of uniqueness exhibits a sharp transition accurately described by
the  $\triangleright$-marked curve $k_{1/\delta}/d^2$ from \eqref{eq:kinv-criterion}, which for the 3 camera case lies under
the $\circ$-marked curve  $k_{opt}$ from \eqref{eq:kopt-criterion}.
}
\label{fig:Prob_Hex_3C}
\end{figure}
\begin{figure}
\begin{center}
\begin{tabular}{cc}
\includegraphics[clip,width=0.45\textwidth]{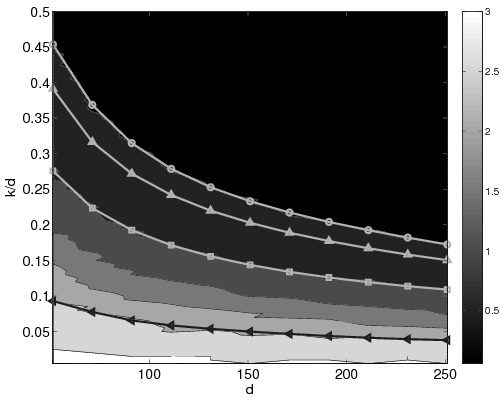} &
\includegraphics[clip,width=0.45\textwidth]{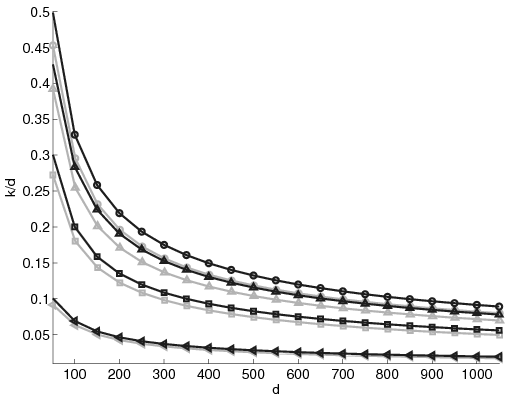}\\
\end{tabular}
\end{center}
\caption{Left: The analytically derived curves from Section \ref{sec:Hex3C} correctly follow the
contour lines of the average fraction of reduced systems determined empirically
as a function of $d$ and relative sparsity $k/d$.
From bottom to top: $\tilde k_{\delta}/d$ ($\triangleleft$-marked curve),
$k_{crit}/d$ (${\square}$-marked curve), $k_{opt}/d$ ($\circ$-marked curve) and
$k_{1/\delta}/d$ ($\triangleleft$-marked curve). Right:
These curves are plotted again (light gray) for a wider range and compared to
the analogous curves for the geometry in Fig. \ref{fig:3_cam}, right, square area, 2 orthogonal, one diagonal
projecting direction.}
\label{fig:HexvsSquare}
\end{figure}

\section{Numerical Experiments and Discussion}\label{sec:num}

In this section we relate the previously derived bounds 
on the required sparsity that guarantee unique nonnegative or binary $k$-sparse 
solutions to numerical experiments. In analogy to \cite{DonTan05} we assess the so called
\emph{phase transition} $\rho$ as a function of $d$, which
is reciprocally proportional to the undersampling ratio $\frac{m}{n}\in(0,1)$.
We vary $d$, build a specific matrix projection matrix $A$ along with 
its perturbed version $\tilde A$ and consider the sparsity as a fraction of $d$ in 2D or $d^2$ in 3D,
respectively, thus $k=\rho d^{D-1}$ , with $\rho\in (0,4)$ and $D\in\{2,3\}$. 
This phase transition $\rho(d)$ indicates the necessary relative sparsity
to recover a $k$-sparse solution with overwhelming probability. More precisely, 
if  $\|x\|_0\le\rho(d) \cdot d^{D-1}$, then with
overwhelming probability a random $k$-sparse nonnegative (or binary) vector $x^*$ is the unique
solution in $\calF_+:=\{x \colon Ax= Ax^*, x\ge 0\}$ or 
$\calF_{ 0,1}:=\{x \colon Ax= Ax^*, x\in[0,1]^n\}$, 
respectively. Uniqueness can be ''verified'' by minimizing and maximizing the same objective $f^\top x$
over $\calF_+$ or $\calF_{ 0,1}$, respectively. If the minimizers coincide for
several random vectors $f$ we claim uniqueness. The resulting linear programs we solved
by a standard LP solver \footnote{MOSEK, \url{http://www.mosek.com/}}.
As shown e.g. in Fig. \ref{fig:sliceA3D} and confirmed by all
our numerical experiments the threshold for a unique nonnegative solution
and a unique  $0/1$-bounded solution are quite close, especially for high values of $d$.

We note that $\tilde A$ has the same sparsity structure as $A$, but random entries 
drawn from the standard uniform distribution on the open interval 
$(0.9,1.1)$. 

Then for $\rho\in[0, 1]$ a $\rho d^{D-1}$-sparse  nonnegative or binary vector
was generated to compute the right hand side measurement vector
and for each $(d,\rho)$-point 50 random problem instances
were generated. A threshold-effect is clearly visible in all figures exhibiting 
parameter regions where the probability of exact reconstruction is
close to one and it is much stronger for the perturbed systems. The results are in excellent
agreement with the derived analytical thresholds. We refer to the 
figure captions for detailed explanations and stress that a threshold-effect is clearly visible in all figures 
exhibiting parameter regions where the probability of exact reconstruction is close to one.

\begin{figure}
\begin{center}
\begin{tabular}{cc}
\includegraphics[clip,width=0.45\textwidth]{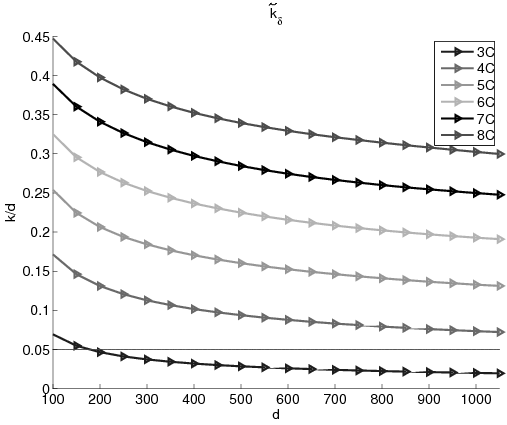} &
\includegraphics[clip,width=0.45\textwidth]{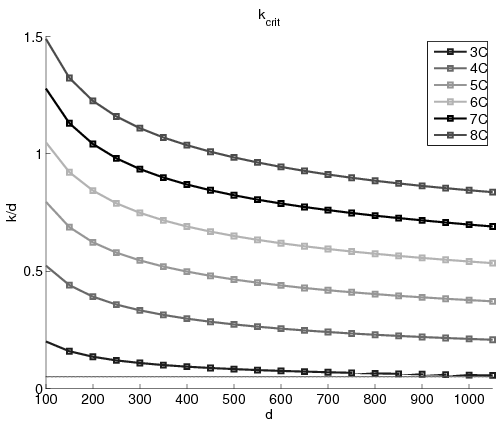}\\
\includegraphics[clip,width=0.45\textwidth]{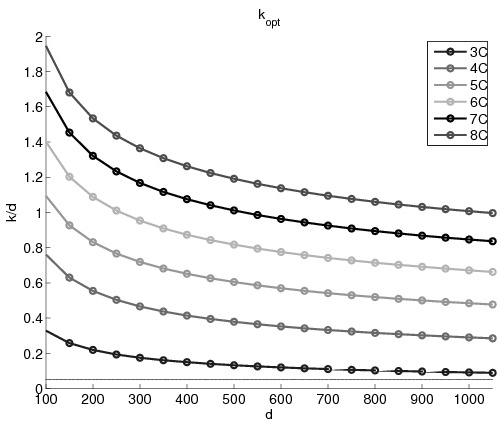} &
\includegraphics[clip,width=0.45\textwidth]{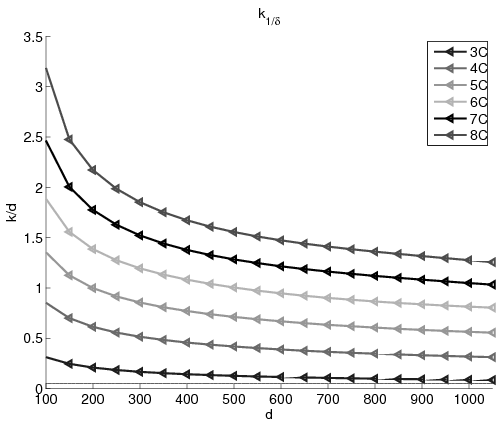}\\
\end{tabular}
\end{center}
\caption{Empirical relative critical curves in for 3 to 8 cameras in 2D.
Top left: $\tilde k_{\delta}/d$ ($\triangleright$-marked curve).
Top right: $k_{crit}/d$ (${\square}$-marked curve).
Bottom left: $k_{opt}/d$ ($\circ$-marked curve).
Bottom right: $k_{1/\delta}/d$ ($\triangleleft$-marked curve). Perfect recovery is possible
for sparsity levels below $k_{1/\delta}$ for perturbed systems.}
\label{fig:k_MoreCams}
\end{figure}
\begin{figure}
\begin{center}
\begin{tabular}{cc}
\includegraphics[clip,width=0.45\textwidth]{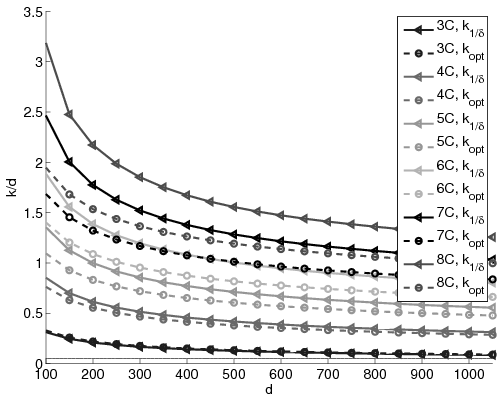} &
\includegraphics[clip,width=0.45\textwidth]{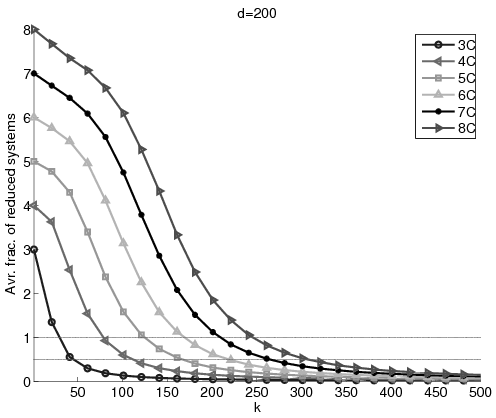}
\end{tabular}
\end{center}
\caption{Left:  $k_{opt}/d$ ($\circ$-marked curve) along with 
$k_{1/\delta}/d$ ($\triangleleft$-marked curve). For 3 cameras $k_{1/\delta}/d$ lies below
$k_{opt}/d$, but starting with 5 cameras  $k_{1/\delta}$ significantly outperforms $k_{opt}$.
This shows the fundamental difference between the considered 0/1-matrices and random matrices
underlying a symmetrical distribution with respect to the origin. For random matrices
recovery of $k$-sparse positive or binary vectors with sparsity levels beyond $k_{opt}$
would be \emph{impossible}. Right: For $d=200$ the 6 curves depict the average ratio of 
$m_{red}(k)/n_{red}(k)$ as a function of sparsity $k$ for 3 to 8 cameras from bottom to top.}
\label{fig:AvrFracMoreCams}
\end{figure}
\begin{figure}
\begin{center}
\begin{tabular}{cc}
\includegraphics[clip,width=0.45\textwidth]{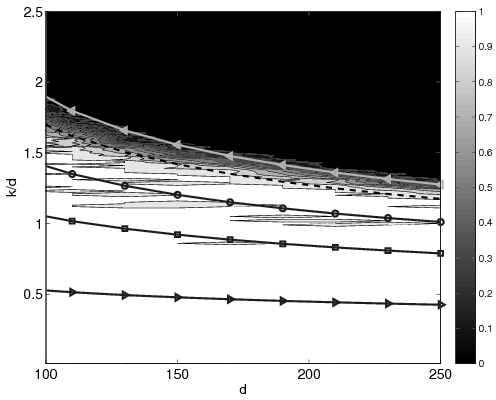} &
\includegraphics[clip,width=0.45\textwidth]{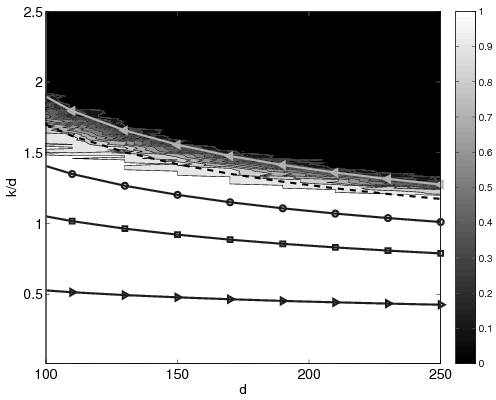}
\end{tabular}
\end{center}
\caption{Success and failure empirical phase transitions for the 2D, 6 cameras case.
Probability of uniqueness in $[0,1]^n$ of a $k=\rho d^2$ sparse binary vector
for unperturbed (left) and  perturbed matrices (right), along with
$\tilde k_{\delta}/d^2$ ($\triangleright$-marked curve),
$k_{crit}/d^2$ (${\square}$-marked curve),
$k_{opt}/d^2$ ($\circ$-marked curve),
and  $k_{1/\delta}/d^2$ ($\triangleleft$-marked curve) from bottom to top.
The dashed line depicts $k/d^2$, with $k$  solving $m_{red}(k)=\frac{2}{\ell} n_{red}(k)$,
which accurately follows the border of the highly success area \emph{for all} considered number
of cameras, $3, 4 \dots 8$.
Recovery is possible \emph{beyond} $k_{opt}$, accurately described by $k_{1/\delta}$.
In the 6 cameras case there is no evident performance boost for perturbed systems, since
the columns of reduced systems are most likely to be in general position for both perturbed and 
unperturbed systems. However, in the perturbed case, recovery is more stable.}
\label{fig:Square_6C}
\end{figure}

\subsection{2D: Hexagonal Volume and 3 Cameras}

In this 2D, 3-cameras case, we generated $A$ according to the geometry 
described in Section \ref{sec:Hex3C}, Fig. \ref{fig:3_cam}, left. We considered $d\in\{51,71,91,\dots ,251\}$
and varied the sparsity $k=\rho d$, by varying $\rho$ in $(0,0.5)$ with constant stepsize $0.01$.
The obtained empirical phase transitions are depicted in \ref{fig:Prob_Hex_3C}.
Fig. \ref{fig:AvrFracMoreCams} additionally shows how the analytically determined critical curves
compare to the critical curves obtained empirically for the geometry in Fig. \ref{fig:3_cam}, right.
The $(4d-1)\times d^2$ projection matrix corresponding to a square area and three projecting directions
has similar reconstruction properties and the critical curves  slightly change by factor $\frac{10}{9}$.

\begin{figure}
\begin{center}
\begin{tabular}{cc}
\includegraphics[clip,width=0.4\textwidth]{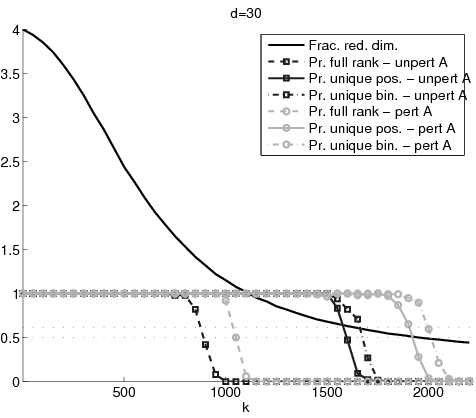} &
\includegraphics[clip,width=0.4\textwidth]{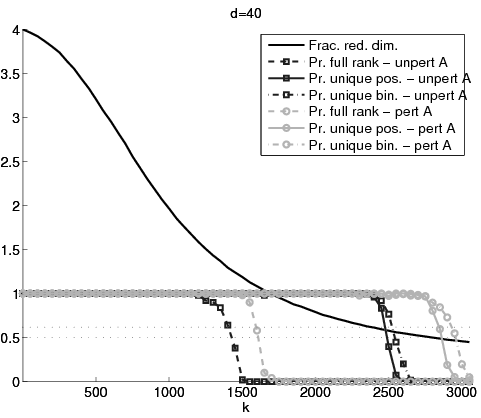}\\
\end{tabular}
\end{center}
\caption{
Recovery via the unperturbed matrix from Section \ref{sec:4C3D}, Fig. \ref{fig:4C3D}
(dark gray $\square$-marked curves), for $d=30$ (left) and $d=40$ (right)
versus the perturbed counterpart (light gray $\circ$-marked curves).
The dashed lines depict the empirical probability (500 trials) that reduced systems are overdetermined and
of full rank. The solid lines show the probability that a $k$ sparse nonnegative vector is unique. The dash-dot curves  shows the probability that a $k$ sparse binary solution is the  unique solution of in $[0,1]^n$. Additional information like binarity gives only a slight performance
boost, as $d$ increases. The curve $\tilde k_{\delta}$ \eqref{eq:kdelta-criterion} correctly predicts that
500 ($d=30$) and 787 ($d=40$) particles are reconstructed with high probability via the unperturbed systems.  
Up to the  sparsity level 1136 ($d=30$) and 1714 ($d=40$) perturbed systems are overdetermined and of full rank 
according to $k_{crit}$ \eqref{eq:kcrit-criterion}. Moreover, perturbed systems have a 
unique $k$ sparse solution if $k$ lies in between $k_{opt}$ - 2028 ($d=30$) and 2856 ($d=40$)
and $k_{1/\delta}$ - 2136 ($d=30$) and 3128 ($d=40$).}
\label{fig:sliceA3D}
\end{figure}

\begin{figure}
\begin{center}
\includegraphics[clip,width=0.45\textwidth]{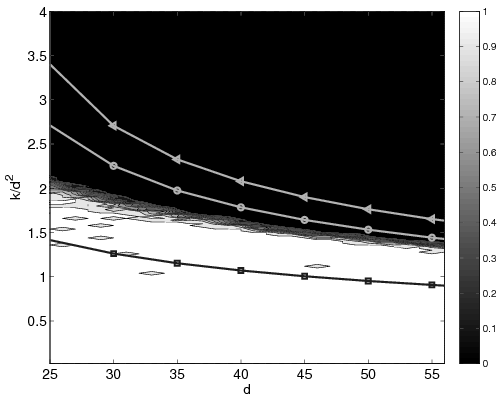}
\includegraphics[clip,width=0.45\textwidth]{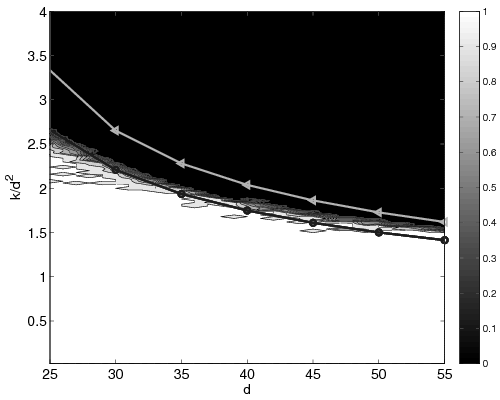}\\
\includegraphics[clip,width=0.45\textwidth]{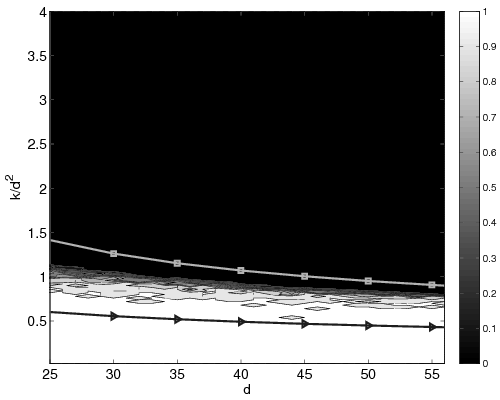} 
\end{center}
\caption{ Success and failure empirical phase transitions for the 3D, 4 cameras case, from Section \ref{sec:4C3D}, Fig. \ref{fig:4C3D}.
Probability of uniqueness in $[0,1]^n$ of a $k=\rho d^2$ sparse binary vector
for unperturbed (top left) and  perturbed matrices (top right).
The $\triangleright$-marked curve depicts $\tilde k_{\delta}/d^2$ \eqref{eq:kdelta-criterion}, 
the $\square$-marked curve $k_{crit}/d^2$ 
\eqref{eq:kcrit-criterion}, the $\circ$-marked curve  $k_{opt}$ \eqref{eq:kopt-criterion}
and the  $\triangleleft$-marked curve $k_{1/\delta}/d^2$ \eqref{eq:kinv-criterion}.
In case of the perturbed matrix $\tilde A$ exact recovery is possible \emph{beyond} 
$k_{opt}/d^2$. Moreover $k_{1/\delta}/d^2$ follows most accurately the empirical phase transition for perturbed systems for high values of $d$.
The empirical probability that the reduced unperturbed matrices are overdetermined
and of full rank (bottom figure), exhibits a threshold in between the estimated relative critical sparsity level $k_{\delta}$ and $k_{crit}$. This explains the performance boost for perturbed systems.}
\label{fig:A4C3D}
\end{figure}

\subsection{2D: Square Volume and 3 to 8 Cameras}

In our theoretical analysis in the previous sections we derived  the expected number of nonzero rows $N_R(k)$ induced by the
$k$-sparse vector along with the number $N_C(k)$ of ''active'' cells which cannot be empty.
This can be done also empirically, compare Fig. \ref{fig:HexvsSquare}, left. 
We have done this in 2D up to 8 projecting directions and obtained \emph{empirically} the critical curves
$k_\delta$, $k_{crit}$, $k_{opt}$ and $k_{1/\delta}$. 
To generate the curves we varied $k/d\in(0,4)$ by a constant stepsize $0.01$
and $d\in\{50,100, \cdots, 1050\}$ and generated for each point $(k/d,d)$
500 problem instances. Further we determined the contour lines of $N_R(k/d,d)/N_C(k/d,d)$
corresponding to the levels $\{\delta\ell,1,0.5,\frac{1+\delta}{\ell}\}$.
The relative sparsity curves $k_\delta/d$, $k_{crit}/d$, $k_{opt}/d$ and $k_{1/\delta}/d$
are plotted in Fig. \ref{fig:k_MoreCams} and accurately follow the empirical recovery thresholds, as shown e.g. in
Fig. \ref{fig:Square_6C} for 6 cameras. Fig. \ref{fig:AvrFracMoreCams} shows 
$N_R(k,200)/N_C(k,200)$ for varying number of cameras.
The projection angles we chosen such that the intersection with all cells is constant, 
yielding binary projection matrices after scaling. Each camera resolution differs with different angle.
We summarize the used parameters in Table \ref{tab:1}.
\begin{table}%
\begin{tabular}{|c|c|c|r|}
	\hline
\# &  $m$    &   $n$  &  projection angles  \\
\hline
3rd camera  & $4d-1$ & $d^2$ & $0^\circ, 90^\circ, 45^\circ$\\
4th camera  & $6d-2$ & $d^2$ & $0^\circ, 90^\circ, \mp 45^\circ$\\
5th camera & $7d + \lfloor \frac{d}{2}\rfloor-2$ & $d^2$ & $0^\circ, 90^\circ, \mp 45^\circ, \arctan(2)$\\
6th camera  & $8d + 2\lfloor \frac{d}{2}\rfloor-2$ & $d^2$ & $0^\circ, 90^\circ, \mp 45^\circ, \mp \arctan(2)$\\
7th camera  & $9d + 3\lfloor \frac{d}{2}\rfloor-2$& $d^2$ & $0^\circ, 90^\circ, \mp 45^\circ, \mp \arctan(2), \arctan(0.5)$\\
8th camera  & $10d + 4\lfloor \frac{d}{2}\rfloor-2$& $d^2$ & $0^\circ, 90^\circ, \mp 45^\circ, \mp \arctan(2), \mp \arctan(0.5)$\\
	\hline
\end{tabular}
\caption{Dimensions of full projection matrices.}
\label{tab:1}
\end{table}

\subsection{3D: 4 Cameras}

In 3D we consider the matrix from Section \ref{sec:4C3D}, Fig. \ref{fig:4C3D}, vary
$d\in\{25,26,\dots,55\}$ and $k=\rho d^2$ by varying $\rho\in(0,4)$ with stepsize $0.01$.
For larger values of $d$ the empirical thresholds follow accurately the estimated curves.
Note that for $d=55$, $A$ is a $48180\times 166375$ matrix.

\section{Conclusion}
The new measurement paradigm of compressed sensing  seeks to capture
the ''essential'' aspects of a high-dimensional but sparse object using as few measurements as
possible by randomization.
Provided that the measurements satisfy certain properties nonnegative sparse signals
can be reconstructed exactly from a surprisingly small number of samples.
Moreover, there exist precise thresholds on sparsity such that for any nonnegative solution that is 
sparser than the threshold is also the unique nonnegative solution of the underlying linear system.
Tomographic projection matrices do not satisfy the conditions which allow 
applying these results on the image reconstruction problem even if the sought solution is very sparse.
However, we analytically showed in the present work that there are thresholds on sparsity 
depending on the numbers of measurements, below which uniqueness is guaranteed and recovery will
succeed and above which it fails with high probability. When recovery succeeds it
yields perfect reconstructions, without any ghost-particles.

\bibliographystyle{plain}

\end{document}